\documentclass[11pt]{amsart}
\usepackage{amsmath,amsthm,amssymb}
\usepackage{mathrsfs}
\usepackage{tikz}
\usepackage{tikz-cd}
\usetikzlibrary{arrows.meta}
\usepackage{subcaption} 
\usepackage[margin=1in,footskip=0.25in]{geometry}
\usepackage{enumerate}
\usepackage{multicol, multirow,mathrsfs}
\usepackage{hyperref}
\newtheorem{theorem}{Theorem}[section]
\newtheorem{lemma}[theorem]{Lemma}

\newtheorem{proposition}[theorem]{Proposition}
\newtheorem{corollary}[theorem]{Corollary}
\newtheorem*{TheoremA}{Theorem A}
\newtheorem*{TheoremA'}{Theorem A'}
\newtheorem*{TheoremB}{Theorem B}
\newtheorem*{TheoremC}{Theorem C}

\theoremstyle{remark}
\newtheorem{remark}[theorem]{Remark}

\newtheorem*{claim*}{Claim}

\newcommand{\C}{\ensuremath{\mathbb{C}}}
\newcommand{\Z}{\mathbb{Z}}
\newcommand{\F}{\mathbb{F}}
\newcommand{\R}{\ensuremath{\mathbb{R}}}

\renewcommand{\O}{\ensuremath{\mathbb{O}}}
\newcommand{\s}[1]{\ensuremath{\mathsf{#1}}}
\newcommand{\g}[1]{\ensuremath{\mathfrak{#1}}}

\DeclareMathOperator{\tr}{tr}

\DeclareMathOperator{\Ad}{Ad}

\DeclareMathOperator{\ad}{ad}
\DeclareMathOperator{\Exp}{Exp}

\DeclareMathOperator{\Isom}{Isom}

\DeclareMathOperator{\spann}{span}

\DeclareMathOperator{\Ric}{Ric}

\DeclareMathOperator{\rank}{rank}

\renewcommand{\H}{\ensuremath{\mathbb{H}}}

\newcommand{\SO}{\ensuremath{\mathsf{SO}}}

\newcommand{\Un}{\ensuremath{\mathsf{U}}}
\newcommand{\SU}{\ensuremath{\mathsf{SU}}}
\newcommand{\Sp}{\ensuremath{\mathsf{Sp}}}

\newcommand{\Spin}{\ensuremath{\mathsf{Spin}}}

\begin{document}
	\title[The index of symmetry and homogeneous fibrations]{The index of symmetry and homogeneous fibrations}
	\author[A.~Cidre-D\'iaz]{\'Angel~Cidre-D\'iaz}
	\address{Citmaga, Universidade de Santiago de Compostela, Spain}
	\email{angel.cidre.diaz@usc.es}
	\author[C.~Olmos]{Carlos E.~Olmos}
	\address{CIEM, CONICET, Argentina}
	\email{carlos.olmos@unc.edu.ar}
	\author[A.~Rodr\'iguez-V\'azquez]{Alberto~Rodr\'iguez-V\'azquez}
	\address{Université Libre de Bruxelles, Belgium}
	\email{alberto.rodriguez.vazquez@ulb.be}
	
	\begin{abstract}
		We develop new tools to compute the index of symmetry in the context of homogeneous fibrations.  As a consequence of our results, we determine the index of symmetry of every homogeneous space diffeomorphic to a compact rank-one symmetric space, and of every homogeneous $\mathsf{S}^1$--bundle over a compact irreducible symmetric space. Moreover, we construct irreducible homogeneous metrics whose leaves of symmetry are symmetric spaces of arbitrarily large rank.
	\end{abstract}

	\thanks{The first and third authors have been supported by grant PID2022-138988NB-I00 funded by MICIU/AEI/10.13039/501100011033 (Spain), and by grants ED431F 2020/04 and ED431C 2023/31 (Xunta de Galicia, Spain). The first author also acknowledges the support of an FPU fellowship. The second author was supported by the visiting program of PPGM, UFSCar, and partially supported by CIEM-CONICET. The third author was further supported by the Horizon Europe research and innovation programme under Marie Sklodowska-Curie Actions with grant agreement 101149711 - HOLYFLOW}

	\maketitle
	\vspace{-0.5ex}
	\section{Introduction}
	\label{sec:intro}
	A Riemannian manifold $M$ is \emph{symmetric} if the geodesic reflection at each point is a global isometry. Alternatively, a complete Riemannian manifold $M$ is {symmetric} if and only if for every $p \in M$, the tangent space $T_p M$ is spanned by \emph{infinitesimal transvections} at $p$, that is, by Killing vector fields $X$ satisfying $(\nabla X)_p = 0$. This motivates the introduction of the so-called \emph{index of symmetry} of a Riemannian manifold, which is defined as 
	\[
	i_{\mathfrak{s}}(M):= \min_{p \in M} \{\dim(\mathfrak{s}_p)\}, \quad \text{where} \enspace \mathfrak{s}_p = \{X_p \colon X \, \text{is Killing}\, \ \text{and} \ (\nabla X)_p = 0\}.
	\] 
	The index of symmetry of a complete Riemannian manifold $M$ reaches its maximum value (equal to $\dim(M)$) if and only if $M$ is symmetric, and thus, it provides a geometric invariant that, in a certain sense, measures how far the Riemannian metric on $M$ is from being symmetric.
	
	Several structural results concerning the index of symmetry have been established in~\cite{ORT}. For instance, the subspaces spanned by infinitesimal transvections integrate to a foliation of $M$ whose leaves (which are not necessarily of the same dimension) are totally geodesic symmetric submanifolds. The study of the index of symmetry has been carried out mainly in the setting of compact homogeneous (non-symmetric) spaces $M=\mathsf{G}/\mathsf{H}$. In particular, the associated foliation is a $\mathsf{G}$--invariant integrable regular distribution, see~\cite{ORT}. Within this framework, some results are known; for example, certain structural results for the index of symmetry of Kähler flag manifolds, see~\cite{podesta}; or the classification of homogeneous spaces admitting metrics with coindex of symmetry less than or equal to $4$, see~\cite{BOR} and \cite{reggiani}. All in all, there has not been much work on the computation of the index of symmetry for arbitrary $\mathsf{G}$--invariant metrics, and the known results are restricted to naturally reductive metrics, see for instance \cite[Theorem~A and Theorem~B]{ORT}.
	
	The most fundamental examples of symmetric metrics are those defined on the \emph{compact rank-one symmetric spaces} (abbreviated as {CROSSes}). The simply connected ones are:
	\[\s{S}^{n}, \quad \C \mathsf{P}^{n}, \quad \mathbb{H} \mathsf{P}^{n}, \quad \mathbb{O} \mathsf{P}^{2}.   \]
	When endowed with their symmetric metrics, the simply connected CROSSes constitute all the {compact, simply connected Riemannian manifolds that are homogeneous and isotropic}, i.e., those whose geometry is the same at every point and in every direction, see~\cite{Szabo2}. They also exhaust all {simply connected compact harmonic manifolds}, confirming the {Lichnerowicz conjecture} in the compact case, see~\cite{Szabo1}. From the viewpoint of geometric structures, CROSSes serve as the {fundamental models of K\"ahler and quaternionic-K\"ahler geometry}. For instance, any compact K\"ahler or quaternionic-K\"ahler manifold with positive sectional curvature is isometric to $\mathbb{C}\mathsf{P}^n$ or $\mathbb{H}\mathsf{P}^n$, respectively, see~\cite{berger}. In fact,  in the study of the interplay between topology and positive curvature, CROSSes provide the {canonical examples} on which many rigidity results are based, see e.g.\,\cite{wilking}, and appear also in {curvature pinching theorems}, see e.g.\,\cite{brendle-schoen}.
	
	The main goal of this article is to establish general structural results for computing the index of symmetry of general $\mathsf{G}$--invariant metrics (not necessarily normal homogeneous or naturally reductive). A first application of these tools is the computation of the index of symmetry of all homogeneous spaces diffeomorphic to CROSSes.
	The non-symmetric homogeneous spaces diffeomorphic to a compact rank-one symmetric space (up to scaling) are isometric to one in the following list:
	\[\mathsf{S}^n_{\F,\tau},\enspace \text{with $\F\in\{\C,\H,\O \}$}; \quad \mathsf{S}^{4n+3}_{\tau_1,\tau_2,\tau_3},\quad \text{and} \quad \C\mathsf{P}^{2n+1}_{\tau}, \]
	where $\tau>0$ and  $\tau_1\ge\tau_2\ge\tau_3>0$. Geometrically, these parameters deform the standard round metric of curvature $1$ on the spheres, and the standard symmetric metric on $\mathbb{C}\mathsf{P}^{2n+1}$, by rescaling the metric along the fibers of specific homogeneous fibrations. More precisely, the metrics on the spheres are obtained by rescaling the round metric along the fibers of the Hopf fibrations. On the other hand, the metrics $\mathbb{C}\mathsf{P}^{2n+1}_{\tau}$ are constructed by scaling the symmetric metric along the $\mathsf{S}^2$--fibers of the twistor fibration over the quaternionic projective space $\mathbb{H}\mathsf{P}^n$. We refer to~\S\ref{subsec:homCROSSes} for the construction of these homogeneous Riemannian metrics. Our first main result of this article is the following:
	\begin{TheoremA}\label{TheoremB}
		Let $M$ be a simply connected non-symmetric homogeneous space diffeomorphic to a compact rank-one symmetric space. Then its index of symmetry is stated as follows:
		\smallskip
		\begin{itemize}
			\item If $M=\mathsf{S}^{n}_{\mathbb{F},\tau}$, then 
			\[
			\begin{aligned}
				i_{\mathfrak{s}}(M) =  \left\{ \begin{array}{ll}
					1 \quad &\textrm{if $\mathbb{F}=\mathbb{C}$,} \\ \smallskip
					3 \quad &\textrm{if $\mathbb{F}=\mathbb{H}$ and $\tau = \frac{1}{2}$,}  \\ \smallskip
					7 \quad &\textrm{if $\mathbb{F}=\mathbb{O}$ and $\tau = \frac{1}{2}$,} \\ \smallskip
					0 \quad &\textrm{otherwise}.
				\end{array} \right.
			\end{aligned}
			\] \smallskip
			\item If $M=\mathsf{S}^{4n+3}_{\tau_1,\tau_2,\tau_3}$, then
			\[\hspace{1.95cm}
			\begin{aligned}
				i_{\mathfrak{s}}(M) =  \left\{ \begin{array}{ll}
					1 \quad &\textrm{if $\tau_1=\tau_2+\tau_3$, and $\tau_1\neq\tau_2\neq\tau_3$,} \\[1ex]
					1 \quad &\textrm{if $\tau_3\neq\tau_1=\tau_2\in\left\{\tfrac{1}{2},1\right\}$,}  \\[1ex]
					1 \quad &\textrm{if $\tau_1\neq\tau_2=\tau_3\in\left\{\tfrac{1}{2},1\right\}$,}\\[1ex]
					3 \quad &\textrm{if $\tau_{1}=\tau_{2}=\tau_{3}=\tfrac{1}{2}$,} \\[1ex]
					0 \quad &\textrm{otherwise.} \\[1ex]
				\end{array} \right.
			\end{aligned}
			\]
			\smallskip

			\item If $M=
			\mathbb{C}\mathsf{P}^{2n+1}_{\tau}$, then $i_{\mathfrak{s}}(M)=0$, for all $\tau>0$. 
		\end{itemize}
	\end{TheoremA}
	Notice that the metrics with $\tau=1$ or $\tau_1=\tau_2=\tau_3=1$ are excluded from the discussion as they are symmetric. Furthermore, although Theorem~\ref{TheoremB} is explicitly stated for simply connected manifolds, this classification completely determines the index of symmetry for all homogeneous spaces diffeomorphic to CROSSes. The only non-simply connected cases are the real projective spaces, and as we will detail in Remark~\ref{Remark::RHn}, their homogeneous metrics are exactly those induced by the spheres, and this covering preserves the index of symmetry. Moreover, as a consequence of~\cite{verdiani-ziller}, the only non-symmetric metrics with positive curvature and non-trivial index of symmetry in non-symmetric homogeneous spaces diffeomorphic to CROSSes are: $\mathsf{S}^{2n+1}_{\mathbb{C},\tau}$ with $\tau<4/3$ and $\tau\neq 1$, $\mathsf{S}^{4n+3}_{\mathbb{H},1/2}$ and $\mathsf{S}^{15}_{\mathbb{O},1/2}$.
	
	It is also worth noting that all these homogeneous metrics fall within the broad class of invariant metrics on homogeneous fibrations.
	
	More precisely, homogeneous fibrations are constructed by considering a triple of compact Lie groups $\mathsf{H}<\mathsf{K}<\mathsf{G}$, and the $\mathsf{G}$--equivariant submersion $\pi$ given~by
	\[F=\mathsf{K}/\mathsf{H}\rightarrow M=\mathsf{G}/\mathsf{H} \xrightarrow{\pi} B=\mathsf{G}/\mathsf{K}, \quad  \text{where $\pi(g\mathsf{H})=g\mathsf{K}$}. \]
	A natural choice of metrics in this context are the so-called canonical variations of a normal homogeneous metric of the total space. To be more specific, these are the metrics  $\langle\cdot,\cdot\rangle_{\lambda}$ on $\mathsf{G}/\mathsf{H}$ constructed by rescaling the length of the fibers with respect to a normal homogeneous metric of $\mathsf{G}/\mathsf{H}$. Canonical variations of homogeneous metrics have long been a powerful tool for constructing manifolds with special curvature properties. For instance, the first non-symmetric, positively curved examples, the {Berger spheres}~$\mathsf{S}^{2n+1}_{\C,\tau}$, with $\tau<4/3$, have a metric of this type, see~\cite{Berger1961,verdiani-ziller}. Building on this idea, Wallach showed that suitable canonical variations of normal homogeneous metrics yield positively curved metrics on the flag manifolds $\mathsf{SU}_3/\mathsf{T}^2$, $\mathsf{Sp}_3/\mathsf{Sp}_1^3$, and $\mathsf{F}_4/\mathsf{Spin}_8$, see~\cite{Wallach1972}. In the realm of Einstein manifolds, canonical variations have been also quite fruitful. Using these metrics, Jensen constructed Einstein examples on certain total spaces over symmetric bases such as the Stiefel manifolds $V_2(\mathbb{R}^n)$ or~$\mathsf{S}^{4n+3}_{\H,\tau}$ with $\tau=\tfrac{1}{2n+3}$, see~\cite{Jensen1973}; Ziller found non-symmetric Einstein metrics in $\C\mathsf{P}_{\tau}^{2n+1}$ with $\tau=\tfrac{1}{n+1}$, see~\cite{ziller1982}; and more generally, together with Wang, he extended this framework to produce Einstein homogeneous metrics on torus bundles over symmetric spaces, see~\cite{WangZiller1990}.
	
	In this work, we study the extension of infinitesimal transvections from the fiber to the total space of a homogeneous fibration endowed with a canonical variation $\langle\cdot,\cdot\rangle_\lambda$ of the normal homogeneous metric, see Equation~\eqref{eq:canvar} for the definition of $\langle\cdot,\cdot\rangle_\lambda$. A basic motivating observation is that non-trivial compact leaves of symmetry can only occur when such a homogeneous fibration exists, see Lemma~\ref{lemma:hmax} and Remark~\ref{rem:closedleaf}. This naturally leads to the investigation of the relation between the index of symmetry and homogeneous fibrations. In particular, we firstly focus on the situation in which the fiber $F$ is one-dimensional, showing that there always exists an infinitesimal transvection on the total space that is tangent to the fiber, see Proposition~\ref{prop:s1bundlelow}. As an application of this, we are able to prove the following result.
	
	\begin{TheoremB}
		Let $B$ be a symmetric space of compact type equipped with an irreducible symmetric metric not locally isometric to the oriented real Grassmannian of $2$-planes, and let $\mathsf{G}=\mathrm{Isom}^{0}(B)$.
		
		Then every $\mathsf{G}$--invariant non-symmetric metric of a homogeneous $\mathsf{S}^{1}$--fibration over $B$ has index of symmetry equal to $1$.
	\end{TheoremB}
	It turns out that all of the metrics described in Theorem~B are homothetic to Sasakian homogeneous metrics and give rise to natural examples of the celebrated Boothby–Wang fibration over a Kähler manifold (in this case Hermitian symmetric); see~\cite{boothbywang} and~\cite[\S 7.9]{Blair}.
	When $B=\mathsf{SO}_{n+2}/(\mathsf{SO}_n\times\mathsf{SO}_2)$, that is, $B$ is  the oriented Grassmannian of $2$-planes, the corresponding $\mathsf{S}^1$--fibration is diffeomorphic to $T_1\mathsf{S}^{n+1}$, the unit tangent bundle of $\mathsf{S}^{n+1}$. It was observed in~\cite[Example~6.3]{ORT} that there is an $\mathsf{SO}_{n+2}$--invariant metric on $T_1\mathsf{S}^{n+1}$ whose leaves of symmetry have dimension $n$, implying that the above result is sharp. However, the argument presented there relies on considerations that are not made fully explicit. In this article, we compute the index of symmetry of every $\mathsf{SO}_{n+2}$--invariant metric on $T_1\mathsf{S}^{n+1}$, see Theorem~\ref{th:indstiefel}.

	More generally, in this article we also consider the case in which the fiber $F$ of a homogeneous fibration has dimension greater than one.
	In this setting, we identify a distinguished canonical variation of the normal homogeneous metric, namely $\langle\cdot,\cdot\rangle_\lambda$ with $\lambda=2$, for which infinitesimal transvections of the fiber extend to infinitesimal transvections of the total space, see Proposition~\ref{prop:verticaltrans}. This phenomenon is specific to the metric corresponding to $\lambda=2$, as shown in Proposition~\ref{prop:verticaltransconversedimge2}.  Moreover, under certain technical assumptions, we observe that the leaf of symmetry coincides with the fiber if and only if $\lambda=2$, see Corollary~\ref{cor:indnotextension}.
	
	For homogeneous spaces diffeomorphic to CROSSes, whenever the index of symmetry is not trivial, the distribution of symmetry gives rise, in all known cases, to spherical leaves. In fact, the symmetric spaces that have so far been realized as leaves of symmetry of irreducible homogeneous spaces in the literature are essentially limited to compact Lie groups or spheres. A natural question, therefore, is whether every irreducible symmetric space can be realized as the leaf of symmetry of some irreducible homogeneous space. As a first step to address this question, we apply the results developed in Subsection~\ref{subsec:transext} to show that irreducible compact symmetric spaces of arbitrarily large rank can indeed be realized as leaves of symmetry.
	
	\begin{TheoremC}Let $M=\mathsf{G}/\mathsf{H}$ be a homogeneous space listed in Table~\ref{table:theoremb}, and consider the homogeneous fibration $\pi \colon \mathsf{G}/\mathsf{H} \to \mathsf{G}/\mathsf{K}$ over the base space $B = \mathsf{G}/\mathsf{K}$, with fiber $F = \mathsf{K}/\mathsf{H}$. Let $\langle\cdot,\cdot\rangle$ be a $\mathsf{G}$-invariant metric on $M$. Then the index of symmetry of $\langle\cdot,\cdot\rangle$ is nontrivial if and only if $\langle\cdot,\cdot\rangle$ is homothetic to the canonical variation $\langle\cdot,\cdot\rangle_{\lambda}$ with $\lambda=2$.
		
		Moreover, for the metric $\langle\cdot,\cdot\rangle_{2}$, the corresponding fibrations, consisting of the total space $M$, the base space $B$, and its leaf of symmetry $F$, are listed in Table~\ref{table:theoremb}.
		
		\begin{table}[h!]
			
			\centering
			
			\renewcommand{\arraystretch}{1.3} 
			
			\begin{tabular}{|c|c|c|c|}
				
				\hline
				
				$M = \mathsf{G}/\mathsf{H}$ & $B = \mathsf{G}/\mathsf{K}$ & $F = \mathsf{K}/\mathsf{H}$ & Comments \\
				
				\hline
				
				$\mathsf{SO}_{m+2n}/(\mathsf{SO}_m\times \mathsf{U}_{n})$ 
				
				& $\mathsf{SO}_{m+2n}/(\mathsf{SO}_m \times \mathsf{SO}_{2n})$ 
				
				& $\mathsf{SO}_{2n}/\mathsf{U}_{n}$ 
				
				& $n\geq 2$ \\
				
				\hline
				
				$\mathsf{SU}_{n+m}/\mathsf{S}(\mathsf{SO}_{n}\times \mathsf{U}_{1}\times \mathsf{U}_{m})$  
				
				& $\mathsf{SU}_{n+m}/\mathsf{S}(\mathsf{U}_{n} \times \mathsf{U}_{m})$
				
				& $\mathsf{SU}_{n}/\mathsf{SO}_{n}$ 
				
				& $n\geq 2$ \\
				
				\hline
				
				$\mathsf{Sp}_{m+n}/(\mathsf{Sp}_m\times \mathsf{U}_{n})$  
				
				& $\mathsf{Sp}_{m+n}/(\mathsf{Sp}_m \times \mathsf{Sp}_{n})$
				
				& $\mathsf{Sp}_{n}/\mathsf{U}_n$ 
				
				& $n\geq 3$ \\
				
				\hline
				
				$\mathsf{E}_6/(\mathsf{Spin}_7\mathsf{Spin}_3\mathsf{SO}_2)$ 
				
				& $\mathsf{E}_6/(\mathsf{Spin}_{10}\mathsf{SO}_2)$
				
				& $\mathsf{SO}_{10}/(\mathsf{SO}_{7}\times\mathsf{SO}_{3})$
				
				& \\
				
				\hline
				
				$\mathsf{E}_6/(\mathsf{Spin}_5\mathsf{Spin}_5\mathsf{SO}_2)$ 
				
				& $\mathsf{E}_6/(\mathsf{Spin}_{10}\mathsf{SO}_2)$
				
				& $\mathsf{SO}_{10}/(\mathsf{SO}_{5}\times\mathsf{SO}_{5})$
				
				& \\
				
				\hline
				
				$\mathsf{E}_7/(\mathsf{Spin}_7\mathsf{Spin}_5\mathsf{Sp}_1)$ 
				
				& $\mathsf{E}_7/(\mathsf{Spin}_{12}\mathsf{Sp}_1)$
				
				& $\mathsf{SO}_{12}/(\mathsf{SO}_{7}\times\mathsf{SO}_{5})$
				
				& \\
				
				\hline
				
				$\mathsf{E}_7/(\mathsf{Spin}_6\mathsf{Spin}_6\mathsf{Sp}_1)$ 
				
				& $\mathsf{E}_7/(\mathsf{Spin}_{12}\mathsf{Sp}_1)$
				
				& $\mathsf{SO}_{12}/(\mathsf{SO}_{6}\times\mathsf{SO}_{6})$
				
				& \\
				
				\hline
				
				$\mathsf{E}_7/(\mathsf{Sp}_4 \times\mathsf{SO}_2)$ 
				
				& $\mathsf{E}_7/(\mathsf{E}_{6}\times\mathsf{SO}_2)$
				
				& $\mathsf{E}_{6}/\mathsf{Sp}_{4}$ 
				
				&  \\
				
				\hline
				
				$\mathsf{E}_8/(\mathsf{SU}_8 \times\mathsf{Sp}_1)$ 
				
				& $\mathsf{E}_8/(\mathsf{E}_{7}\times \mathsf{Sp}_1)$
				
				& $\mathsf{E}_{7}/\mathsf{SU}_{8}$ 
				
				&  \\
				
				\hline
				
			\end{tabular}
			
			\vspace{0.2cm} 
			
			\caption{Homogeneous fibrations associated with the canonical variation. The canonical variation scales the metric along the fibers $F$, which arise as the leaves of symmetry.}
			
			\label{table:theoremb}
			
		\end{table}
	\end{TheoremC}
	
	\subsection*{Organization of the paper}
	In Section~\ref{sec:prelim}, we collect some standard facts on Riemannian homogeneous spaces (\S\ref{subsec:homspaces}), on the index of symmetry (\S\ref{subsec:indsym}), and we describe homogeneous metrics on spaces diffeomorphic to CROSSes (\S\ref{subsec:homCROSSes}). In Section~\ref{sec:homfibr}, we develop tools for studying the index of symmetry in the context of homogeneous fibrations. In \S\ref{subsec:stiefel} we study the index of symmetry of invariant metrics on $T_1\mathsf{S}^{n+1}$, which may be viewed both as an $\mathsf{S}^1$--fibration and as an $\mathsf{S}^n$--fibration, and already exhibits the phenomena described in \S\ref{subsec:s1fib} and \S\ref{subsec:transext}. In particular, in \S\ref{subsec:structure2isotr}, we establish a structural result for the distribution of symmetry of homogeneous spaces whose isotropy representation decomposes into two inequivalent irreducible submodules; see Theorem~\ref{th:2-isotropy_fibration}.  In \S\ref{subsec:s1fib}, we analyze the index of symmetry of homogeneous $\mathsf{S}^1$--fibrations and provide the proof of Theorem~B. In \S\ref{subsec:transext}, we study necessary and sufficient conditions for extending infinitesimal transvections from the fiber to the total space of a homogeneous fibration; see Propositions~\ref{prop:verticaltrans} and~\ref{prop:verticaltransconversedimge2}. These results are then used to prove Theorem~C.  Finally, in Section~\ref{sec:indexcross}, making use of several results obtained in Section~\ref{sec:homfibr}, we compute the index of symmetry of homogeneous spaces diffeomorphic to CROSSes, thereby completing the proof of Theorem~A.

	\subsection*{Acknowledgements:}
	The authors wish to thank David González-Álvaro and Miguel Domínguez-Vázquez for their valuable comments on
	an earlier draft of this manuscript.  We also thank Jason DeVito and Megan Kerr for helpful conversations and correspondence.

	\section{Preliminaries}
	\label{sec:prelim}
	
	In this section, we collect background results that will be used throughout the article regarding Killing vector fields, homogeneous spaces, the index of symmetry, and homogeneous spaces diffeomorphic to compact rank-one symmetric spaces.
	
	\subsection{Some general facts about Killing fields and homogeneous spaces} \label{subsec:homspaces}
	We introduce some basic facts and notation concerning Killing vector fields and Riemannian homogeneous spaces, see~\cite[Chapter~X]{KN2} for more details.
	
	Let $M$ be an arbitrary Riemannian manifold. We denote by $\mathcal{K}(M)$ the Lie algebra of Killing vector fields of $M$. A convenient formula to compute the Levi-Civita~\cite[Equation~3.2]{ORT} connection is
	\begin{equation}\label{eq:L-C}
		2\langle  \nabla _{X}Y , Z\rangle =
		\langle [X, Y], Z\rangle 
		+ \langle [X, Z] , Y\rangle
		+ \langle [Y, Z] , X\rangle, \quad \text{where  $X,Y,Z\in\mathcal{K}(M)$}.
	\end{equation}

	Let $X\in\mathcal{K}(M)$ and choose a fixed point $p\in M$. The pair $(X_p,  (\nabla X)_p)$, where $(\nabla X)_p$ denotes the skew-symmetric endomorphism of $T_p M$ given by $v\in T_p M\mapsto (\nabla_v X)_p$,   determines uniquely a Killing vector field defined on $M$. Moreover, if $X, Y\in\mathcal{K}(M)$, one can check that $Z=[X,Y]\in\mathcal{K}(M)$ has initial conditions at $p$ given by
	\begin{equation}
		\label{eq:bracketkilling}
		Z_{p}=[X,Y]_p \quad \text{and}\quad (\nabla Z)_{p}=R_{p}(X_{p},Y_{p}) -[(\nabla X)_p,(\nabla Y)_p],
	\end{equation}
	where $R$ denotes the curvature tensor of $M$. 
	
	We also have the following formula that relates the parallel transport, the Levi-Civita connection, and the flow of a Killing vector field; see~\cite[Equation (2.2.1)]{Di-Scala-Olmos-Vittone}. Let us denote by  $\gamma\colon(-\varepsilon,\varepsilon)\rightarrow M$ the integral curve of $X\in\mathcal{K}(M)$ starting at $p\in M$. Then, the following identity holds
	\begin{equation}
		\label{eq:magicolmos}
		\left((\mathcal{P}^{\gamma}_{0,t})^{-1}\circ d\phi_t \right)v=(e^{\nabla X})v \quad \text{for all $v\in T_p M$},
	\end{equation}
	where $\mathcal{P}^{\gamma}_{0,t}$ denotes the parallel transport along $\gamma$, and $\phi_t$ is the flow associated with $X$.
	
	Now let us assume that $M=\mathsf{G}/\mathsf{H}$ is a compact homogeneous space with base point $o\in M$. We consider the canonical connection $\nabla^c$ associated with the reductive decomposition $\mathfrak{g}=\mathfrak{h}\oplus\mathfrak{m}$, which is the unique $\mathsf{G}$--invariant affine connection on $M$ such that
	\begin{equation*}
		\label{eq:nablac}
		(\nabla^c_{X^*}Y^*)_o=(-[X,Y]_{\mathfrak{m}})^*_o, \quad\text{where $X,Y\in\mathfrak{m}$.}
	\end{equation*}
	We recall that the canonical connection $\nabla^c$ has the relevant property that every $\mathsf{G}$--invariant tensor on $M=\mathsf{G}/\mathsf{H}$ is $\nabla^c$-parallel, see~\cite[Proposition~2.7]{KN2}.
	
	From now on, we consider the reductive decomposition  $\mathfrak{g} = \mathfrak{h} \oplus \mathfrak{m}$ induced by taking as $\mathfrak{m}$ the orthogonal complement of $\mathfrak{h}$ with respect to an $\Ad(\mathsf{G})$--invariant inner product $b$ on $\mathfrak{g}$. We identify $\mathfrak m \cong T_{o}(\mathsf{G}/\mathsf{H})$, via $X\in\mathfrak{m}\mapsto
	X^{*}_{o} = \left.\frac{d}{dt}\right|_{t=0} \Exp(tX)\cdot o$, where $X\in\mathfrak{m}$. Recall that the isotropy representation of $M$ can be identified with the adjoint representation of $\mathsf{H}$ on $\mathfrak m$. This induces a splitting $\mathfrak{m} = \mathfrak{m}_{0} \oplus \mathfrak{m}_{1} \oplus \ldots \oplus \mathfrak{m}_{r}$ where $\mathsf{H}$ acts trivially on $\mathfrak{m}_{0}$ and irreducibly on $\mathfrak{m}_{1},\ldots, \mathfrak{m}_{r}$. 
	If all $\mathfrak{m}_i$, where $i\in\{1,\dots,r\}$ are inequivalent $\mathsf{H}$-modules, then any $\mathsf{G}$--invariant metric on $\mathsf{G}/\mathsf{H}$ comes from an $\Ad(\mathsf{H})$--invariant inner product on $\mathfrak{m}$ of the form
	\begin{equation}\label{eq:Homogeneous_Metric}
		\langle X , Y \rangle = h(X_{\mathfrak{m}_0},Y_{\mathfrak{m}_0}) +\sum_{i=1}^{r}\alpha_{i}b(X_{\mathfrak{m}_i},Y_{\mathfrak{m}_i}),
	\end{equation}
	where $(\cdot)_{\mathfrak{m}_i}$ denotes the projection to $\mathfrak{m}_i$, $h$ is an arbitrary positive definite inner product on $\mathfrak{m}_{0}$, and $\alpha_{i}>0$ a constant. In the case where there exist two equivalent irreducible $\mathsf{H}$-submodules $\mathfrak{m}_i$ and $\mathfrak{m}_j$, the decomposition $\mathfrak{m}_i \oplus \mathfrak{m}_j$ is not unique, and we need to introduce additional parameters to account for the possibility that the inner product restricted to $\mathfrak{m}_i \times \mathfrak{m}_j$ may not vanish. There are three possibilities for an irreducible $\mathsf{H}$-module such as $\mathfrak{m}_i$, since by Schur’s lemma the space of $\mathsf{H}$-intertwining operators $\mathrm{End}_{\mathsf{H}}(\mathfrak{m}_i)$ is isomorphic to either $\R$, $\C$, or $\H$. We say that the $\mathsf{H}$-module $\mathfrak{m}_i$ is of \textit{real}, \textit{complex}, or \textit{quaternionic} type, accordingly. In each case, we must introduce $1$, $2$, or $4$ extra parameters to completely determine the moduli space of $\mathsf{G}$--invariant metrics on $\mathsf{G}/\mathsf{H}$.
	
	Let us fix a $\mathsf{G}$--invariant metric $\langle \cdot, \cdot\rangle$ on $M=\mathsf{G}/\mathsf{H}$.   We say that $M=\mathsf{G}/\mathsf{H}$ is a \emph{normal homogeneous Riemannian manifold} if $\langle \cdot, \cdot \rangle$ coincides with the restriction to $\mathfrak{m}\times\mathfrak{m}$ of an $\Ad(\mathsf{G})$--invariant inner product $b$ of $\mathfrak{g}$, where $\mathfrak{m}$ is chosen to be the orthogonal complement of $\mathfrak{h}$ in $\mathfrak{g}$ with respect to $b$. On the other hand, we say that a homogeneous Riemannian manifold $ M = \mathsf{G}/\mathsf{H} $ equipped with a $\mathsf{G}$--invariant metric $\langle \cdot , \cdot \rangle$ is \emph{naturally reductive} if the tensor $U\colon\mathfrak{m}\times\mathfrak{m}\rightarrow\mathfrak{m}$ defined by 
	\begin{equation}\label{eq:U-tensor}
		2 \langle U(X, Y),Z \rangle = \langle [Z,X]_{\mathfrak{m}},Y\rangle + \langle X, [Z,Y]_{\mathfrak{m}}\rangle, \quad \text{for } X,Y,Z\in\mathfrak{m,}\end{equation}
	vanishes identically. In particular, normal homogeneous Riemannian manifolds are naturally reductive since the map $[X,\cdot]\colon\mathfrak{g}\rightarrow\mathfrak{g}$ is an anti-self-adjoint endomorphism of $\mathfrak{g}$ with respect to $b$ for each $X\in\mathfrak{g}$.
	
	A formula for the Levi-Civita connection of Killing vector fields induced by elements of $\mathfrak{m}$ in a Riemannian homogeneous space is
	\begin{equation}
		\label{eq:LCconnectionkillingind}
		(\nabla_{X^*} Y^*)_o=\left(-\frac{1}{2}[X,Y]_{\mathfrak{m}} + U(X,Y) \right)^*_o, \quad \text{where $X,Y\in \mathfrak{m}$.}
	\end{equation}
	Thus, using that the Levi-Civita connection is torsion-free, we deduce
	\begin{equation}
		\label{eq:LCconnectionkillingindarb}
		(\nabla_{X^*} V)_o= (D_{X} \xi)_o^* + [X^*,V]_o \quad \text{for every $X\in\mathfrak{m}$ and every vector field $V$}, 
	\end{equation}
	where $\xi\in\mathfrak{m}$ satisfies that $\xi^*_o=V_o$; and $D:=\nabla-\nabla^c$ is the difference tensor, which can be computed by
	\begin{equation}
		\label{eq:dtensor}
		D_X Y=\frac{1}{2}[X,Y]_{\mathfrak{m}}+ U(X,Y), \quad\text{for every $X,Y\in\mathfrak{m}$.}  
	\end{equation}

	A nice class of homogeneous metrics can be constructed by considering homogeneous fibrations. Given compact Lie groups $\mathsf{H}<\s K<\mathsf{G}$, we can consider the smooth $\mathsf{G}$--equivariant map $\pi\colon \mathsf{G}/\mathsf{H}\rightarrow \mathsf{G}/\mathsf{K}$ given by $\pi(g \s H)=g \s K$.  This is a smooth submersion with fiber diffeomorphic to $\mathsf{K}/\mathsf{H}$, and it is called the \textit{homogeneous fibration} associated with the triple $\mathsf{H}<\s K<\mathsf{G}$.  The $\mathsf{G}$--equivariant map $\pi$ is a Riemannian submersion when  $\mathsf{G}/\mathsf{K}$ and $\mathsf{G}/\mathsf{H}$ are endowed with the normal homogeneous metrics induced by a fixed background $\Ad(\mathsf{G})$--invariant inner product $b$ defined on $\mathfrak{g}$. Let $\mathcal{V} $ be the (autoparallel) vertical distribution, and let $ \mathcal{H} = \mathcal{V}^\perp $ be the horizontal distribution. Since $\mathcal{V}$ and $\mathcal{H}$ are obviously $\mathsf{H}$--invariant, we can make the following identifications $\mathfrak{m}_1:=\mathfrak{h}^\perp\cap\mathfrak{k}\cong \mathcal{V}_{o}$ and $\mathfrak{m}_2:=\mathfrak{k}^\perp\cong \mathcal{H}_{o}$, where $\mathfrak{h}^\perp$ and $\mathfrak{k}^\perp$ denote the orthogonal complement in $\mathfrak{g}$ with respect to $b$. Then, for each $\lambda>0$, we can consider the $\mathsf{G}$--invariant metric on $\mathsf{G}/\mathsf{H}$ induced by the $\Ad(\mathsf{H})$--invariant inner product on $\mathfrak{m}$ given by
	\begin{equation}
		\label{eq:canvar}
		\langle X, Y\rangle_{\lambda}=\lambda\, b(X_{\mathfrak{m}_1},Y_{\mathfrak{m}_1}) + \, b(X_{\mathfrak{m}_2},Y_{\mathfrak{m}_2}),
	\end{equation}
	where $(\cdot)_{\mathfrak{m}_i}$ denotes the orthogonal projection to $\mathfrak{m}_i$ for each $i\in\{1,2\}$.
	
	The metric $\langle \cdot, \cdot\rangle_{\lambda}$ on $M=\mathsf{G}/\mathsf{H}$ is referred to as a \textit{canonical variation} of the normal homogeneous metric on $M=\mathsf{G}/\mathsf{H}$. Observe that when $\mathsf{G}/\mathsf{K}$ is equipped with the normal homogeneous metric induced by $b$, the map $\pi$ is a Riemannian submersion for each metric $\langle \cdot, \cdot\rangle_{\lambda}$ on $M=\mathsf{G}/\mathsf{H}$. Moreover, by Equation~\eqref{eq:LCconnectionkillingind}, when $M$ is equipped with a normal homogeneous metric, that is, when $M$ is endowed with the metric $\langle \cdot,\cdot\rangle_{\lambda}$ with $\lambda = 1$, the curves $\gamma_X(t) = \Exp(tX)\cdot o$, where $X \in \mathfrak{m}_1$, are geodesics of $M$ tangent to the fiber $\pi^{-1}(e \mathsf{K})$. In particular, the fiber $\pi^{-1}(e\mathsf{K})$ is a totally geodesic submanifold of $M$. Indeed, every fiber of $\pi$ is a totally geodesic submanifold of $M$, since it is congruent to $\pi^{-1}(e\mathsf{K})$ by the \s{G}--equivariance of $\pi$. Then, the following result follows by \cite[Lemma~3.2]{dearricott}.
	\begin{lemma}
		\label{lemma:tghomfib}
		Let $\mathsf{H}<\mathsf{K}<\mathsf{G}$ induce a homogeneous fibration as above. Then, the fibers $\mathsf{K}/\mathsf{H}$ are totally geodesic submanifolds of $M=\mathsf{G}/\mathsf{H}$ with respect to every canonical variation $\langle \cdot,\cdot\rangle_{\lambda}$ of a normal homogeneous metric on $M$.
	\end{lemma}
	
	\subsection{The index of symmetry}
	\label{subsec:indsym}
	We recall the notion of the index of symmetry and summarize some of its properties.
	
	A Killing vector field $X \in \mathcal{K}(M)$ is said to be an \emph{infinitesimal transvection} at a point $p \in M$ if its covariant derivative vanishes at $p$, that is, $(\nabla X)_p = 0$. For each point $p\in M$, one can consider all infinitesimal transvections at $p$. The values of these vector fields at $ p $ form a subspace of the tangent space $ T_p M $, which is called the \emph{distribution of symmetry} at $ p $.  The \emph{index of symmetry} of a Riemannian manifold $ M $ is defined as the greatest integer $ k $ such that, at every point $ p \in M $, there exist at least $ k $ linearly independent vectors $ v_1, \ldots, v_k \in T_p M $ and corresponding Killing vector fields $ X_1, \ldots, X_k $ on $ M $ satisfying $ X_i(p) = v_i $ and $ (\nabla X_i)_p = 0 $ for all $ i $. Equivalently, the index of symmetry can be expressed as
	\[
	i_{\mathfrak{s}}(M) := \min_{p \in M} \dim(\mathfrak{s}_p),
	\]
	where $ \mathfrak{s}_p $ denotes the distribution of symmetry at the point $ p $. The index of symmetry of a complete Riemannian manifold $ M $ is maximal, that is, $ i_{\mathfrak{s}}(M) = \dim(M) $, if and only if $ M $ is symmetric, see e.g.~\cite[Section~2]{BOR}. Therefore, the index of symmetry measures, in some way, how far a given Riemannian manifold is from being a symmetric space.
	
	The index of symmetry has been studied primarily for compact normal homogeneous and naturally reductive homogeneous spaces. For instance, in \cite{ORT} it is shown that if $M = \mathsf{G}/\mathsf{H}$ with $\mathsf{H}=\mathsf{G}_{o}$ is a simply connected, compact, normal homogeneous space that is irreducible and non-symmetric, then the distribution of symmetry coincides with the subspace of $\mathcal{K}(M)$ consisting of Killing fields whose infinitesimal generators are fixed by the isotropy. That is,
	\begin{equation}
		\label{eq:sym_isotropy}
		\mathfrak{s}_{o} = \{ X^{*}_{o} \colon \Ad(h)X = X \ \text{for all} \ h \in \mathsf{H} \}.
	\end{equation}
	The same result holds for naturally reductive homogeneous spaces $ M = \mathsf{G}/\mathsf{H} $, but under the additional assumption that the Lie group $ \s G $ has Lie algebra $ \mathfrak{g} = [\mathfrak{m}, \mathfrak{m}] + \mathfrak{m}$, see \cite[Theorem B]{ORT}. Moreover, consider an arbitrary Riemannian homogeneous space $M =\mathsf{G}/\mathsf{H}$, where $o = e\mathsf{H}$ and $\mathsf{H}=\mathsf{G}_{o}$ is the isotropy group. Then, the distribution of symmetry is a smooth distribution whose integral leaves are totally geodesic and globally symmetric (see~\cite[Section~3]{ORT}). We denote by $L(q)$ the integral leaf of the distribution of symmetry passing through $q \in M$, which we shall call the \emph{leaf of symmetry} at $q$. Let us focus on the origin $o$ and let $\mathsf{K}$ be the global stabilizer of the leaf of symmetry $L(o)$, that is,
	\[
	\mathsf{K}=\{g \in \mathsf{G} : g \cdot L(o) = L(o)\}. 
	\]
	Then, its identity component $\mathsf{K}^{0}$ is a connected Lie group, and $\mathsf{H} \cap \mathsf{K}^{0}$ is a closed Lie subgroup of $\mathsf{K}^{0}$. Thus, the quotient $\mathsf{K}^{0}/(\mathsf{H} \cap \mathsf{K}^{0})$ is a homogeneous presentation of the leaf of symmetry~$L(o)$. 
	Moreover, if the distribution of symmetry is non-trivial, then $\mathsf{K}^{0}$ is a connected Lie (possibly not closed) subgroup  which must properly contain $\mathsf{H} \cap \mathsf{K}^{0}$. This implies that if the Riemannian metric on $M$ is non-symmetric and $\mathsf{H}$ is a maximal connected Lie subgroup of $\mathsf{G}$, then the distribution of symmetry is trivial. Note that, since $\mathsf{H}$ is connected, it is entirely contained in $\mathsf{K}^{0}$, and thus $\mathsf{H} \cap \mathsf{K}^{0} = \mathsf{H}$, yielding the following result.
	
	\begin{lemma}
		\label{lemma:hmax}
		Let $M = \mathsf{G}/\mathsf{H}$ be a homogeneous space where $\mathsf{H}<\mathsf{G}$ is a maximal connected Lie subgroup. Then, we have that $i_{\mathfrak{s}}(M)=0$ for every $\mathsf{G}$--invariant non-symmetric metric on $M$. 
	\end{lemma}
	
	\begin{remark}
		\label{rem:closedleaf}
		Observe that if the leaf of symmetry is closed, then $\mathsf{K}$ is a closed Lie subgroup of~$\mathsf{G}$. Therefore, if $\mathsf{G}$ is a compact Lie group, since $\mathsf{H}<\mathsf{K}$ (because $\mathsf{H}\cdot L(o)=L(o)$ by~\cite[p.~615]{ORT}), we obtain a homogeneous fibration induced by the subgroups $\mathsf{H}<\mathsf{K}<\mathsf{G}$. 
		However, it is an open question whether there are examples of homogeneous spaces with non-closed leaves of symmetry, even in the compact setting. In Theorem~\ref{th:2-isotropy_fibration}, we actually prove that the leaves of symmetry are indeed closed under certain additional hypotheses (namely, being a compact homogeneous space with two inequivalent irreducible isotropy modules).
	\end{remark}

	\subsection{The homogeneous spaces diffeomorphic to CROSSes}
	\label{subsec:homCROSSes}
	The classification of transitive actions of connected compact Lie groups on spheres is summarized in the table below. This was obtained by Borel, Montgomery, and Samelson, see~\cite{borel,MontgomerySamelson}.
	\begin{table}[h!]
		\centering
		\begin{tabular}{|c|c|c|c|c|}
			\hline
			& $\mathsf{G}$ & $\mathsf{H}$ & $\dim \mathsf{G}/\mathsf{H}$ & Isotropy representation  \\
			\hline
			(i) & $\mathsf{SO}_{n+1}$ & $\mathsf{SO}_n$ & $n$ & Irreducible \\
			(ii) & $\mathsf{SU}_{n+1}$ &  $\mathsf{SU}_n$ & $2n+1$ & $\mathfrak{m} = \mathfrak{m}_0 \oplus \mathfrak{m}_1$  \\
			(iii) & $\mathsf{U}_{n+1}$ & $\mathsf{U}_n$ & $2n+1$ & $\mathfrak{m} = \mathfrak{m}_0 \oplus \mathfrak{m}_1$  \\
			(iv) & $\mathsf{Sp}_{n+1}$ & $\mathsf{Sp}_n$ & $4n+3$ & $\mathfrak{m} = \mathfrak{m}_0 \oplus \mathfrak{m}_1$  \\
			(v) & $\mathsf{Sp}_{n+1}\mathsf{Sp}_1$ & $\mathsf{Sp}_n\mathsf{Sp}_1$ & $4n+3$ & $\mathfrak{m} = \mathfrak{m}_1 \oplus \mathfrak{m}_2$  \\
			(vi) & $\mathsf{Sp}_{n+1}\mathsf{U}_1$ & $\mathsf{Sp}_n\mathsf{U}_1$ & $4n+3$ & $\mathfrak{m} = \mathfrak{m}_0 \oplus \mathfrak{m}_1 \oplus \mathfrak{m}_2$   \\
			(vii) & $\mathsf{Spin}_9$ & $\mathsf{Spin}_7$ & $15$ & $\mathfrak{m} = \mathfrak{m}_1 \oplus \mathfrak{m}_2$   \\
			(viii) & $\mathsf{Spin}_7$ & $\mathsf{G}_2$ & $7$ & Irreducible \\
			(ix) & $\mathsf{G}_2$ & $\mathsf{SU}_3$ & $6$ & Irreducible  \\
			\hline
		\end{tabular}\\[2ex]
		\caption{List of connected compact Lie groups $\mathsf{G}$ acting effectively and transitively on spheres up to local isomorphism, together with the decomposition of their isotropy representations into irreducible submodules. The submodules $\mathfrak{m}_0$  are trivial, that is, they are fixed by the $\mathsf{H}$--action. Moreover, in this list we denote by $\mathsf{Sp}_{n+1}\mathsf{Sp}_1$ and $\mathsf{Sp}_{n+1}\mathsf{U}_1$    the quotients of $\mathsf{Sp}_{n+1}\times \mathsf{Sp}_1$ and $\mathsf{Sp}_{n+1}\times \mathsf{U}_1$ by $\{\pm(\mathrm{Id},1) \}\cong \Z_2$, respectively.  }
		\label{table:homspheres}
	\end{table}
	
	By Schur's lemma, every sphere presented as a homogeneous space $\mathsf{G}/\mathsf{H}$ with irreducible isotropy admits a unique $\mathsf{G}$--invariant metric up to scaling, the round one. The cases  \textup{(ii)} and \textup{(iii)} yield the same family of $1$-parametric $\mathsf{G}$--invariant metrics up to scaling. Furthermore, we have the following inclusion relations for these families of $\mathsf{G}$--invariant metrics (see~\cite[\S 1]{ziller1982}):
	\[\textup{(v)}\subset \textup{(vi)}\subset \textup{(iv)}. \]
	
	\begin{remark}\label{Remark::RHn}
		The only non-simply connected homogeneous spaces diffeomorphic to a CROSS are the real projective spaces $\mathbb{R}\mathrm{P}^n \cong \mathsf{S}^n/\mathbb{Z}_2$. Clearly, any homogeneous metric on $\mathsf{S}^n$ descends to $\mathbb{R}\mathrm{P}^n$, since the action of its isometry group commutes with $\mathbb{Z}_2 \cong \{\pm \mathrm{Id}\}$. Conversely, any homogeneous metric on $\mathbb{R}\mathrm{P}^n$ pulls back to a homogeneous metric on its universal cover $\mathsf{S}^n$; indeed, the fundamental Killing fields lift globally and generate a transitive isometry group with the exact same Lie algebra. Consequently, the infinitesimal transvections at a point in $\mathsf{S}^n$ correspond bijectively to the infinitesimal transvections at its projection in $\mathbb{R}\mathrm{P}^n$. Therefore, the study of the index of symmetry for homogeneous spaces diffeomorphic to real projective spaces reduces identically to the case of spheres.
	\end{remark}

	\subsubsection{Hopf-Berger spheres} 
	Now we consider the $\mathsf{G}$--invariant metrics of spheres corresponding to the families \textup{(iii)}, \textup{(v)}, and \textup{(vii)}. In order to construct them, we take the total spaces of one of the Hopf fibrations, namely \[\mathsf{S}^1\rightarrow \mathsf{S}^{2n+1}\rightarrow \C \s P^n,\qquad\mathsf{S}^3\rightarrow \mathsf{S}^{4n+3}\rightarrow \H\s  P^n,\qquad \mathsf{S}^7\rightarrow \mathsf{S}^{15}\rightarrow \mathbb{O} \mathsf{P}^1  \]
	and we endow them with the Riemannian metric obtained from rescaling  the  metric tensor of the unit round sphere  by a factor $\tau>0$ in the vertical directions.  We denote such Riemannian homogeneous spaces by  $\mathsf{S}_{\C,\tau}^{2n+1}$, $\mathsf{S}_{\H,\tau}^{4n+3}$, and $\mathsf{S}_{\mathbb{O},\tau}^{15}$, depending on whether the Hopf fibration under consideration is the complex, the quaternionic, or the octonionic one, respectively; see~\cite{ORV} for more details. The following lemma is well known, see, for instance,~\cite{ARV_thesis} for a proof.
	\begin{lemma}\label{lem:Isom_Hopf_Berger}
		Let $\mathsf{S}^{n}_{\mathbb{F},\tau}$, $\tau\neq1$, be a Hopf-Berger sphere. Then:
		\[
		\begin{aligned}
			\Isom(\mathsf{S}^{n}_{\mathbb{F},\tau}) =  \left\{ \begin{array}{ll}
				\Un_{k+1}\rtimes \mathbb{Z}_{2} \quad &\textrm{if $\mathbb{F}=\mathbb{C}$ and $n=2k+1$,} \\ \smallskip
				\Sp_{k+1}\Sp_{1} \quad &\textrm{if $\mathbb{F}=\mathbb{H}$ and $n=4k+3$,} \\ \smallskip
				\Spin_{9} \quad &\textrm{if $\mathbb{F}=\mathbb{O}$ and $n=15$.} \\
			\end{array} \right.
		\end{aligned}
		\]
	\end{lemma}
	
	\begin{remark}
		\label{rem:normalhomtau}
		The $\mathsf{G}$--invariant metrics corresponding to the Hopf-Berger metrics on the spheres $\mathsf{S}^{2n+1}_{\C,\tau}=\mathsf{SU}_{n+1}/\mathsf{SU}_n$, $\mathsf{S}^{4n+3}_{\H,\tau}=\mathsf{Sp}_{n+1}/\mathsf{Sp}_n$, and $\mathsf{S}^{15}_{\mathbb{O},\tau}=\mathsf{Spin}_{9}/\mathsf{Spin}_7$ correspond to the $\mathsf{H}$--invariant inner products $\langle\cdot,\cdot\rangle_{\F,\tau}$. Here, $\F\in\{\C,\H,\O\}$ and $b$ denotes a given $\Ad(\mathsf{G})$--invariant inner product on the corresponding Lie algebra $\mathfrak{g}$ as follows:
		\begin{equation}
			\begin{aligned}
				\langle X,Y\rangle_{\C,\tau}:=& \tfrac{2n}{n+1}\tau b(X_{\mathfrak{m}_0}, Y_{\mathfrak{m}_0}) + b(X_{\mathfrak{m}_1}, Y_{\mathfrak{m}_1}), & &\text{where $b(X,Y)=-\frac{1}{2}\tr(X Y)$},\\
				\langle X,Y\rangle_{\H,\tau}:=& 2\tau b(X_{\mathfrak{m}_0}, Y_{\mathfrak{m}_0}) + b(X_{\mathfrak{m}_1}, Y_{\mathfrak{m}_1}),&  &\text{where $b(X,Y)=-\frac{1}{2}\mathrm{Re}\tr_{\H}(X Y)$},\\
				\langle X,Y\rangle_{\mathbb{O},\tau}:=& 4\tau b(X_{\mathfrak{m}_0}, Y_{\mathfrak{m}_0}) + b(X_{\mathfrak{m}_1}, Y_{\mathfrak{m}_1}),&  &\text{where $b(X,Y)=-\frac{1}{8}\tr(X Y)$}.
			\end{aligned}
		\end{equation}
		Thus, Hopf-Berger metrics are canonical variations of a normal homogeneous metric and, more concretely, each normal homogeneous metric is realized for $\tau=\tfrac{n+1}{2n}$, $\tau=\tfrac{1}{2}$, and $\tau=\tfrac{1}{4}$ for each $\F\in\{\C,\H,\O\}$, respectively.
	\end{remark}
	
	\subsubsection{The  $\mathsf{Sp}_{n+1}$--invariant metrics on spheres} 
	\label{subsubsec:Spn-inv}
	Let us focus on the family of $\mathsf{Sp}_{n+1}$--invariant metrics on $\mathsf{S}^{4n+3}$ given by~\textup{(iv)}.  We define $\mathfrak{g}=\mathfrak{sp}_{n+1}$ and we consider the reductive decomposition $\mathfrak{g}=\mathfrak{h}\oplus\mathfrak{m}$, with $\mathfrak{m}:=\mathfrak{m}_0\oplus\mathfrak{m}_1$, where
	\begin{equation}
		\label{eq:reddecspn-inv}
		\mathfrak{h}=\left(
		\begin{array}{c|c}
			Z & 0 \\
			\hline
			0 & 0
		\end{array}
		\right), \quad \mathfrak{m}_0=\left(
		\begin{array}{c|c}
			0 & 0 \\
			\hline
			0 & \mathrm{Im}(x)
		\end{array}
		\right), \quad \mathfrak{m}_1=\left(
		\begin{array}{c|c}
			0 & v \\
			\hline
			-v^* & 0
		\end{array}
		\right),
	\end{equation}
	where $x\in\H$, $v\in\H^n$ and $Z\in\mathfrak{sp}_n$. We denote by $X_1$, $X_2$ and $X_3$ the vertical Killing vector fields induced by vectors of $\mathfrak{m}_0$, where $x=i$, $x=j$ and $x=k$; respectively.
	
	Then, we denote by $\mathsf{S}^{4n+3}_{\tau_1,\tau_2,\tau_3}$ the Riemannian manifold obtained by considering the total space $\mathsf{S}^{4n+3}$ of the quaternionic Hopf fibration $\mathsf{S}^3\rightarrow \mathsf{S}^{4n+3}\rightarrow \H\mathsf{P}^n$  where we rescaled its round unit metric in the vertical directions $X_1$, $X_2$ and $X_3$ by the positive parameters $\tau_1,\tau_2$ and $\tau_3$, respectively. From now on, we will assume, without loss of generality, that $\tau_{1}\geq \tau_{2}\geq \tau_{3}>0$.
	Then, the corresponding $\Ad(\mathsf{H})$--invariant inner product on $\mathfrak{m}$ is given by
	\begin{equation}
		\label{eq:spnmetric}
		\langle \cdot, \cdot \rangle_{\tau_1,\tau_2,\tau_3}= 2 \sum_{i=1}^{3}\tau_i b(\cdot,\cdot)_{\rvert \R X_i \times \R X_i}+ b(\cdot,\cdot)_{\rvert\mathfrak{m}_1\times\mathfrak{m}_1},
	\end{equation}
	where $b(X,Y)=-\tfrac{1}{2}\mathrm{Re} \tr_{\H}(XY)$, with $X,Y\in\mathfrak{sp}_{n+1}$ is  a certain multiple of the Killing form of $\mathfrak{g}=\mathfrak{sp}_{n+1}$.  Indeed, one can check that for $\tau_i=1$ for all $i\in\{1,2,3\}$, we get the unit round metric on $\mathsf{S}^{4n+3}$,  see~\cite{ziller1982} for a computation. Moreover, $\langle \cdot, \cdot \rangle_{\tau_1,\tau_2,\tau_3}$ exhaust all possible $\mathsf{Sp}_{n+1}$--invariant metrics on $\mathsf{S}^{4n+3}$ up to homothety, see~\cite{ziller1982}.

	By \cite[Theorem 2.2]{kollrosstams} and Table~\ref{table:homspheres}, we have the  diagram of inclusions depicted in~Figure~\ref{fig:spn-inclusions}, where each one of the inclusions is maximal. 	Then, we can deduce the following.
	\begin{lemma}\label{Lemma::Isometry_Group_Spheres} The  connected component containing the identity of the group of isometries of $\mathsf{S}^{4n+3}_{\tau_{1},\tau_{2},\tau_{3}}$ with $\tau_{1}\geq\tau_{2}\geq \tau_{3}>0$ is given by one of the following cases below: 
		\begin{enumerate}[{\rm (i)}]
			\item If $\tau_{1}=\tau_{2}=\tau_{3}=1$, then $\mathsf{S}^{4n+3}_{\tau_{1},\tau_{2},\tau_{3}}$ is the round sphere, and thus $\Isom^{0}(\mathsf{S}^{4n+3}_{\tau_{1},\tau_{2},\tau_{3}})=\SO_{4n+4}$.
			\item If $\tau_{1}=\tau_{2}=\tau_{3}\neq 1$, then $\mathsf{S}^{4n+3}_{\tau_{1},\tau_{2},\tau_{3}}=\mathsf{S}^{4n+3}_{\H,\tau}$, and thus $\Isom^{0}(\mathsf{S}^{4n+3}_{\tau_{1},\tau_{2},\tau_{3}})=\Sp_{n+1}\Sp_{1}$.
			\item  If $\tau_1=\tau_2=1$ and $\tau_3=\tau< 1$; or $\tau_1=\tau> 1$ and $\tau_2=\tau_3=1$, then $\mathsf{S}^{4n+3}_{\tau_{1},\tau_{2},\tau_{3}}=\mathsf{S}^{4n+3}_{\C,\tau}$, and thus, $\Isom^{0}(\mathsf{S}^{4n+3}_{\tau_{1},\tau_{2},\tau_{3}})=\Un_{2n+2}$.
			\item  If $1\neq \tau_1=\tau_2\neq\tau_3$ or $\tau_1\neq\tau_2=\tau_3\neq 1$, then $\Isom^{0}(\mathsf{S}^{4n+3}_{\tau_{1},\tau_{2},\tau_{3}})=\mathsf{Sp}_{n+1}\mathsf{U}_1$.
			\item If $\tau_1>\tau_2>\tau_3$, then $\Isom^{0}(\mathsf{S}^{4n+3}_{\tau_{1},\tau_{2},\tau_{3}})=\Sp_{n+1}$.
		\end{enumerate}
	\end{lemma}
	\begin{figure}[h]
		\centering
		\begin{tikzcd}
			& \mathsf{SO}_{4n+4}   & \\
			\mathsf{U}_{2n+2} \arrow[ur] & & \mathsf{Sp}_{n+1}\mathsf{Sp}_1 \arrow[ul]\\
			\mathsf{SU}_{2n+2} \arrow[u] & & \mathsf{Sp}_{n+1}\mathsf{U}_1 \arrow[u] \arrow[ull] \\
			& \mathsf{Sp}_{n+1} \arrow[ul] \arrow[ur]   &
		\end{tikzcd}
		\caption{Connected subgroups acting effectively and transitively on $\mathsf{S}^{4n+3}$ containing $\mathsf{Sp}_{n+1}$.}
		\label{fig:spn-inclusions}
	\end{figure}
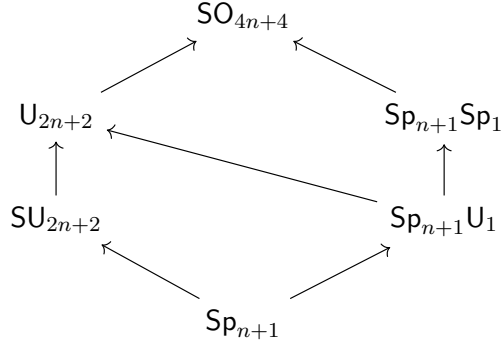
	
	\subsubsection{$\mathsf{Sp}_{n+1}\mathsf{U}_1$--invariant metrics on spheres}
	Now, we will focus on those metrics corresponding to case (vi) in Lemma~\ref{Lemma::Isometry_Group_Spheres}. To ensure consistency with the conventions established we maintain the ordering assumption $\tau_{1}\geq \tau_{2}\geq \tau_{3}>0$. Under this assumption, the condition defining case (vi) distinguishes between two subcases:
	\begin{itemize}
		\item[(A)] $\tau_{1}>\tau_{2}=\tau_{3}\neq 1$. In this scenario, the isometry group is extended by a factor $\Un_{1}$, acting via $e^{it}\cdot p=pe^{-it}$, where $t \in \R$ and $p\in\mathsf{S}^{4n+3}$.
		\item[(B)] $1\neq\tau_{1}=\tau_{2}>\tau_{3}$. Here, the isometry group is extended by a factor $\Un_{1}$, acting via $e^{kt}\cdot p=pe^{-kt}$, where $t \in \R$ and $p\in\mathsf{S}^{4n+3}$.
	\end{itemize}
	
	This distinction is geometrically significant: normal homogeneous metrics appear exclusively in Case (B), whereas Case (A) contains naturally reductive metrics but none of them is normal homogeneous. We will present an expression for the corresponding inner products on $\mathfrak{m}$. For each $q\in\mathrm{Im}(\H)$, we define the elements of $\mathfrak{gl}_{n+2}(\H)$ given by
	\begin{equation}
		\label{eq:defxyq}
		V(q) := \left(\begin{array}{c|c}
			q  & 0 \\ \hline
			0 & 0
		\end{array}\right), \quad  W(q) := \left(\begin{array}{c|c}
			0  & 0 \\ \hline
			0 & q
		\end{array}\right).
	\end{equation}
	We construct the Lie subalgebra $\mathfrak{g}(q):=\mathfrak{g}_1\oplus\mathfrak{g}_2(q)$ of $\mathfrak{gl}_{n+2}(\H)$, where $\mathfrak{g}_1$ is the subalgebra of $\mathfrak{gl}_{n+2}(\H)$ isomorphic to  $\mathfrak{sp}_{n+1}$ which centralizes $V(q)$, and $\mathfrak{g}_2(q)$ is the one-dimensional subspace spanned by $V(q)$. We consider the subalgebra $\mathfrak{h}(q)\subset \mathfrak{g}(q)$ defined by $\mathfrak{h}(q):=\mathfrak{h}_{1}\oplus \mathfrak{h}_{2}(q)$, where
	$\mathfrak{h}_1$ is the subalgebra of $\mathfrak{g}_1$ isomorphic to $\mathfrak{sp}_n$ which centralizes $W(q)$; and $\mathfrak{h}_2(q):=\spann\{V(q)+W(q)\}$.  Notice that $\mathfrak{g}_1$ and $\mathfrak{h}_1$ are independent of the choice of $q\in\mathrm{Im}(\H)$.
	
	Let us fix the background metric $b(X,Y)=-\tfrac{1}{2}\mathrm{Re} \tr_{\H}(XY)$ in $\mathfrak{gl}_{n+2}(\H)$. Then $b$ also induces an $\Ad(\mathsf{G}(q))$--invariant metric, where $\mathsf{G}(q)\cong\mathsf{Sp}_{n+1}\times\mathsf{U}_1$ is the connected subgroup with Lie subalgebra $\mathfrak{g}(q)$ in $\mathfrak{gl}_{n+2}(\H)$.
	We define the subspaces of $\mathfrak{m}(q):=\mathfrak{g}(q)\ominus\mathfrak{h}(q)$ given by   $\mathfrak{m}_{0}(q)=\spann\{V(q)-W(q)\}$, and $ \mathfrak{m}_{1}(q):=\spann\{W(q'): q'\in\mathrm{Im}(\H)\ominus\R q\}$. Now we fix the reductive decomposition
	\begin{equation}
		\label{eq:reddec2param}
		\mathfrak{g}(q)=\mathfrak{h}(q)\oplus \mathfrak{m}(q),
	\end{equation}
	where $\mathfrak{m}(q):=\mathfrak{m}_{0}(q)\oplus \mathfrak{m}_{1}(q) \oplus \mathfrak{m}_{2}$ and $\mathfrak{m}_2$ is the orthogonal complement of $\mathfrak{m}_0(q)\oplus\mathfrak{m}_1(q)$ inside $\mathfrak{m}(q)$ with respect to $b$. Observe that $\mathfrak{m}_2$ does not depend on $q\in \mathrm{Im}(\H)$. Then, the metrics (A) and (B) correspond to the following $\mathsf{H}$--invariant inner products:
	\begin{equation}
		\label{eq:spnu1metric}
		\begin{aligned}
			&\text{(A)} & & \langle \cdot , \cdot \rangle_{\tau_{1},\tau,\tau} = 4\tau_{1} b(\cdot,\cdot)_{\rvert \mathfrak{m}_0(i)\times\mathfrak{m}_0(i)} \hspace{0.1cm}+ 2\tau b(\cdot,\cdot)_{\rvert \mathfrak{m}_1(i)\times\mathfrak{m}_1(i)} \hspace{0.1cm}+ b(\cdot,\cdot)_{\rvert\mathfrak{m}_2\times\mathfrak{m}_2},\\[1ex]
			&\text{(B)} & & \langle \cdot , \cdot \rangle_{\tau,\tau,\tau_3} = 4\tau_3 b(\cdot,\cdot)_{{\rvert \mathfrak{m}_0(k)\times\mathfrak{m}_0(k)}} + 2\tau b(\cdot,\cdot)_{\rvert \mathfrak{m}_1(k)\times\mathfrak{m}_1(k)}+ b(\cdot,\cdot)_{\rvert\mathfrak{m}_2\times\mathfrak{m}_2},
		\end{aligned}
	\end{equation}

	\subsubsection{The homogeneous spaces diffeomorphic to complex projective spaces}

	The Lie groups acting transitively on complex projective spaces were classified by Onishchik, see~\cite{Onishchik2}, and we list them in Table~\ref{table:homcpn}.
	
	\begin{table}[h!]
		\centering
		\begin{tabular}{|c|c|c|c|c|c|}
			\hline
			& $\mathsf{G}$ & $\mathsf{H}$ & $\dim \mathsf{G}/\mathsf{H}$ & Isotropy representation & $\s Z(\mathsf{G})$  \\
			\hline
			(i) & $\mathsf{SU}_{n+1}$ & $\mathsf{U}_n$ & $2n$ & Irreducible & $\zeta\, \mathrm{Id},\, \text{where}\, \zeta^{n+1}=1$  \\
			(ii) & $\mathsf{Sp}_{n+1}$ & $\mathsf{Sp}_n\times\mathsf{U}_1$ & $4n+2$ & $\mathfrak{m} = \mathfrak{m}_1 \oplus \mathfrak{m}_2$ & $\pm \mathrm{Id}$  \\
			\hline
		\end{tabular}\\[2ex]
		\caption{List of connected compact Lie groups $\mathsf{G}$ acting transitively and almost effectively on complex projective spaces up to local isomorphism, together with the decomposition of their isotropy representations. The kernel of ineffectiveness of each $\mathsf{G}$--action is precisely the center of $\mathsf{G}$, which is denoted by $\s Z(\mathsf{G})$. }
		\label{table:homcpn}
	\end{table}
	Now we focus on the family of $\mathsf{Sp}_{n+1}$--invariant metrics on $\mathbb{C}\mathsf{P}^{2n+1}$ given by~\textup{(ii)}.  We define $\mathfrak{g}=\mathfrak{sp}_{n+1}$ and we consider the reductive decomposition $\mathfrak{g}=\mathfrak{h}\oplus\mathfrak{m}$, with $\mathfrak{m}:=\mathfrak{m}_1\oplus\mathfrak{m}_2$, where
	\[ \mathfrak{h}=\left(
	\begin{array}{c|c}
		Z & 0 \\
		\hline
		0 & x i
	\end{array}
	\right), \quad \mathfrak{m}_1=\left(
	\begin{array}{c|c}
		0 & 0 \\
		\hline
		0 & y j + z k
	\end{array}
	\right), \quad \mathfrak{m}_2=\left(
	\begin{array}{c|c}
		0 & v \\
		\hline
		-v^* & 0
	\end{array}
	\right),   \]
	where $x,y,z\in\R$, $v\in\H^n$ and $Z$ belongs to $\mathfrak{sp}_n$. Moreover, we have the homogeneous fibration
	\[
	\Sp_1/\Un_{1} =\s{S}^2{\longrightarrow}\, \Sp_{n+1}/(\Sp_{n}\times \Un_{1})=\mathbb{C}\mathsf{P}^{2n+1}\overset{\pi} {\longrightarrow}
	\Sp_{n+1}/(\Sp_{n} \times \Sp_1) = \mathbb{H}\mathsf{P}^n, 
	\]
	where $g\left( \mathsf{Sp}_n\times\mathsf{U}_1 \right) \overset{\pi}\mapsto  g \left(\mathsf{Sp}_n\times\mathsf{Sp}_1\right)$, and $g\in\mathsf{Sp}_{n+1}$.
	On the one hand, the subspace $\mathfrak{m}_{1}$ is identified with the vertical distribution $\mathcal{V}$ at the base point $o=e\mathsf{H} \in \C \mathsf{P}^{2n+1}$, or equivalently, to the tangent space to the (totally geodesic) fiber $\Sp_{1}\cdot o=\pi^{-1}(\pi(o))$ at $o$. On the other hand, $\mathfrak{m}_{2}$ is identified with the horizontal distribution $\mathcal{H}=\mathcal{V}^{\perp}$ at $o$.  By Equation~(\ref{eq:Homogeneous_Metric}), up to isometry and scaling, any $\Sp_{n+1}$--invariant metric on $\mathbb{C}\mathsf{P}^{2n+1}$ is induced by an $\Sp_{n}\times \Un_{1}$--invariant inner product $\langle \cdot ,\cdot \rangle$ of the form
	\begin{equation}
		\label{eq:metricCPn}
		\langle\cdot ,\cdot \rangle_{\tau} = 2\tau\, b(\cdot , \cdot)_{\rvert \mathfrak{m}_{1}\times\mathfrak{m}_1} + b(\cdot , \cdot)_{\rvert\mathfrak{m}_{2}\times\mathfrak{m}_2},
	\end{equation}
	where $b(X,Y)=-\tfrac{1}{2}\mathrm{Re} \tr_{\H}(XY)$, for each $X,Y\in\mathfrak{sp}_{n+1}$, is  a certain multiple of the Killing form of $\mathfrak{g}=\mathfrak{sp}_{n+1}$. Indeed, one can check that the induced metric is symmetric if and only if $\tau=1$, with sectional curvature between $1$ and $4$, see~\cite{ziller1982} or \cite[p.~64]{vermeirenthesis}. Thus, as a direct consequence of this discussion and Table~\ref{table:homcpn}, we have the following:
	\begin{lemma}\label{lem:Isometry_Group_Projective}The  identity component of the isometry group of $\mathbb{C}\mathsf{P}^{2n+1}_{\tau}$ with $\tau>0$ is
		\[
		\begin{aligned}
			\Isom^{0}(\mathbb{C}\mathsf{P}^{2n+1}_{\tau}) =  \left\{ \begin{array}{ll}
				\SU_{2n+2}/\Z_{2n+2} \quad &\textrm{if $\tau=1$,} \\ \smallskip
				\Sp_{n+1}/\Delta \mathbb{Z}_{2} \quad &\textrm{otherwise.} \\
			\end{array} \right.
		\end{aligned}
		\]
	\end{lemma}
	\section{The index of symmetry and canonical variation metrics}
	\label{sec:homfibr}
	In this section, we study the index of symmetry in the context of canonical variations of the normal homogeneous metric of a  homogeneous fibration. Following the notation in \S~\ref{subsec:homspaces}, let us consider a homogeneous fibration given by $$F=\mathsf{K}/\mathsf{H}\rightarrow M=\mathsf{G}/\mathsf{H}\rightarrow B=\mathsf{G}/\mathsf{K}$$ associated with the inclusions of compact Lie groups $\mathsf{H}<\mathsf{K}<\mathsf{G}$. Then, we can consider a canonical variation of the normal homogeneous metric on $\mathsf{G}/\mathsf{H}$ given by
	\begin{equation}
		\label{eq:lambdametric}
		\langle X, Y\rangle_{\lambda}:=\lambda\, b(X_{\mathfrak{m}_1},Y_{\mathfrak{m}_1}) + b(X_{\mathfrak{m}_2},Y_{\mathfrak{m}_2}) \quad \text{for $\lambda>0$},
	\end{equation}
	where $\mathfrak{m}_1:=\mathfrak{h}^\perp\cap\mathfrak{k}\cong\mathcal{V}_o$ and $\mathfrak{m}_2=\mathfrak{k}^\perp\cong\mathcal{H}_o$.
	
	\subsection{The index of symmetry of the unit tangent bundle of $\mathsf{S}^{n+1}$}
	\label{subsec:stiefel}
	Let us consider $T_1 \mathsf{S}^{n+1}$, that is, the unit tangent bundle of $\mathsf{S}^{n+1}$. This is a homogeneous space which can be presented as $M=\mathsf{G}/\mathsf{H}=\mathsf{SO}_{n+2}/\mathsf{SO}_n$. Throughout this subsection, we will assume $n\ge3$ so that $\mathsf{G}$ is simple.
	
	Let us consider the $\Ad(\mathsf{G})$--invariant inner product $b$ on $\mathfrak{g}=\mathfrak{so}_{n+2}$ given by $b(X,Y)=-\tfrac{1}{2}\tr(X Y)$, where $X,Y\in\mathfrak{so}_{n+2}$. We define the subspaces
	\[\mathfrak{m}_0:=\spann\{E_{1,2}\}, \quad \mathfrak{m}_1:=\spann\{E_{1,3}, E_{1,4}, \ldots, E_{1,n+2}\},\quad \mathfrak{m}_2:=\spann\{E_{2,3}, E_{2,4}, \ldots, E_{2,n+2}\}, \]
	where $E_{i,j}$, with $i<j$ denotes the matrix with a $1$ in the position $(i,j)$ and a $-1$ in the position~$(j,i)$. We consider the reductive decomposition of $M$ given by 
	$\mathfrak{g}=\mathfrak{m}\oplus\mathfrak{h}$, where $\mathfrak{m}=\mathfrak{m}_0\oplus\mathfrak{m}_1\oplus\mathfrak{m}_2$, and $\mathfrak{h}\cong \mathfrak{so}_n$ is the orthogonal complement of $\mathfrak{m}$ with respect to $b$ in $\mathfrak{g}$. Then, the map $J= \Ad(\Exp(-\tfrac{\pi}{2}E_{1,2}))$ provides an equivalence of $\mathsf{H}$--representations between $\mathfrak{m}_1$ and $\mathfrak{m}_2$. Notice that one also has that $JX=[X,E_{1,2}]$ for every $X\in\mathfrak{m}_i$, where $i\in\{1,2\}$, and that $J$ restricted to $\mathfrak{m}_1\oplus\mathfrak{m}_2$ satisfies $J^2=-\mathrm{Id}$ and $b(J X,J X)=b(X,X)$ for every $X\in\mathfrak{m}_1\oplus\mathfrak{m}_2$. Moreover, although $\mathfrak{m}_1$ and $\mathfrak{m}_2$ are equivalent $\mathsf{H}$-representations, every $\mathsf{G}$-invariant metric on $M$ can be parametrized up to isometry as follows (see~\cite[p.~120]{kerr}):
	\begin{equation}
		\label{eq:metricstiefel}
		\langle X, Y\rangle_{\lambda, \mu_1, \mu_2}= \lambda b(X_{\mathfrak{m}_0},Y_{\mathfrak{m}_0}) + \mu_1  b(X_{\mathfrak{m}_1},Y_{\mathfrak{m}_1}) + \mu_2 b(X_{\mathfrak{m}_2},Y_{\mathfrak{m}_2}), 
	\end{equation}
	where $\lambda,\mu_1,\mu_2>0$. Using the properties of the map $J$ described above and the Equations~\eqref{eq:dtensor} and~\eqref{eq:LCconnectionkillingind} one obtains the following identities
	\begin{lemma}
		Let $X_{i}, Y_{i},Z_{i}\in\mathfrak{m}_i$ where $i\in\{1,2\}$. Then, at the base point $o\in M$ we have the following equalities:
		\begin{align}
			\langle D_{E_{1,2}} E_{1,2}, E_{1,2}\rangle_{\lambda,\mu_1,\mu_2}&= \langle D_{E_{1,2}} E_{1,2}, X_1\rangle_{\lambda,\mu_1,\mu_2}=\langle D_{E_{1,2}} E_{1,2}, X_2\rangle_{\lambda,\mu_1,\mu_2}\notag\\
			&=\langle D_{X_1} E_{1,2}, E_{1,2}\rangle_{\lambda,\mu_1,\mu_2}=\langle D_{X_1} E_{1,2}, Y_1\rangle_{\lambda,\mu_1,\mu_2}\label{eq:dtensorstiefel0}\\
			&=\langle D_{X_2} E_{1,2}, E_{1,2}\rangle_{\lambda,\mu_1,\mu_2}=\langle D_{X_2} E_{1,2}, Y_2\rangle_{\lambda,\mu_1,\mu_2}=0,\notag\\
			\langle D_{X_i} E_{1,2},X_{i+1}\rangle_{\lambda,\mu_1,\mu_2}&=\frac{\mu_{i+1} -\mu_i +(-1)^{i+1}\lambda}{2}b(J X_1, X_2),\label{eq:dtensorstiefeln0}
		\end{align}
		where $i\in\{1,2\}$ and indices are taken modulo two;
		
		\begin{align}
			\langle \nabla_{E^*_{1,2}} X^*_i, Y^*_i\rangle_{\lambda,\mu_1,\mu_2}&= \langle \nabla_{E^*_{1,2}} X^*_i, E^*_{1,2}\rangle_{\lambda,\mu_1,\mu_2}=\langle \nabla_{E^*_{1,2}} E^*_{1,2}, E^*_{1,2}\rangle_{\lambda,\mu_1,\mu_2}\notag\\
			&=\langle \nabla_{X^*_i} E^*_{1,2}, Y^*_i\rangle_{\lambda,\mu_1,\mu_2}=\langle \nabla_{X^*_i} Y^*_i, E^*_{1,2}\rangle_{\lambda,\mu_1,\mu_2}\label{eq:nablastiefel0}\\
			&=\langle \nabla_{X^*_i} Y^*_i, Z^*_i\rangle_{\lambda,\mu_1,\mu_2}=\langle \nabla_{X^*_{i+1}} X^*_{i}, Y^*_i\rangle_{\lambda,\mu_1,\mu_2}\notag\\
			&=\langle \nabla_{X^*_{i+1}} X^*_{i}, Y^*_{i+1}\rangle_{\lambda,\mu_1,\mu_2}=\langle \nabla_{X^*_{i}} E^*_{1,2}, E^*_{1,2}\rangle_{\lambda,\mu_1,\mu_2}=0,\notag\\
			\langle \nabla_{X^*_i} X_{i+1}^*,E^*_{1,2}\rangle_{\lambda,\mu_1,\mu_2}&=\frac{\mu_i-\mu_{i+1}+(-1)^{i+1}\lambda}{2}b(J X_1, X_2), 	\label{eq:nablastiefel1}\\
			\langle \nabla_{X^*_i} E_{1,2}^*,X^*_{i+1}\rangle_{\lambda,\mu_1,\mu_2}&=(-1)^i \frac{\left(\mu_i + \mu_{i+1}-\lambda\right)}{2}b(J X_1, X_2),\label{eq:nablastiefel2}
		\end{align}
		where $i\in\{1,2\}$ and indices are taken modulo two.
	\end{lemma}

	Now we will determine the Lie algebra of the isometry group of each one of the metrics defined in \eqref{eq:metricstiefel}. We refer the reader to \cite{Onishchik3} or also~\cite[\S  6]{JDV-MDV-DGA-ARV} for the necessary details on Onishchik's extension theory that we will use from now on. 
	\begin{lemma}
		\label{lem:isomstiefel}
		The algebra of Killing vector fields of $(T_1\mathsf{S}^{n+1},\langle\cdot,\cdot\rangle_{\lambda,\mu_1,\mu_2})$ with $n\ge3$ is isomorphic to $\mathfrak{so}_{n+2}$ when $\mu_1\neq\mu_2$ and isomorphic to $\mathfrak{so}_{n+2}\oplus\mathfrak{so}_2$ when $\mu_1=\mu_2$.
	\end{lemma}
	\begin{proof}
		Firstly, notice that  $M=\mathsf{SO}_{n+2}/\mathsf{SO}_{n}$ admits a type~I extension as $\mathsf{Z}_{\mathfrak{g}}(\mathfrak{h})=\mathfrak{m}_0\cong\mathfrak{so}_2$. Moreover, by \cite[Theorem~3.6]{hsiangstiefel}, if $n\not\in\{2,6\}$, then $M$ cannot be decomposed into a product of two homogeneous spaces with positive dimensions, and thus it does not allow type~III extensions. By inspection of Table~\cite[Table~7]{Onishchik3}, one can check that $M$ does not allow type~II extensions.
		
		When $n=6$, $M$ is diffeomorphic to the product $\mathsf{S}^7\times\mathsf{S}^6$.  Now assume for a contradiction that $M$ admits a type~III extension when $n=6$. This implies that $\mathfrak{g}=\mathfrak{so}_8$ would extend to  $\mathfrak{so}_8\oplus\mathfrak{g}_1$, in such a way that the connected Lie subgroup associated with the first/second factor of $\mathfrak{so}_8\oplus \mathfrak{g}_1$ acts transitively on the first/second factor of $M=\mathsf{S}^7\times\mathsf{S}^6$. As every transitive and effective action of a compact connected Lie group on $\mathsf{S}^k$ is the restriction of the standard action of $\mathsf{SO}_{k+1}$ to a certain subgroup of  $\mathsf{SO}_{k+1}$, we deduce that there is an $\mathsf{SO}_8$--invariant Einstein metric on $M=\mathsf{S}^7\times\mathsf{S}^6$ corresponding to the product of the round metrics of $\mathsf{S}^7$ and $\mathsf{S}^6$ with the same scalings. However, Sagle found in~\cite{sagle_einstein} that there is one irreducible $\mathsf{SO}_{n+2}$--invariant  Einstein metric on $M$, and Kerr proved in~\cite{kerr} that this is the only $\mathsf{SO}_{n+2}$--invariant  Einstein metric on $M$ (up to homotheties), leading to a contradiction with the existence of type III extensions. Consequently, as $\mathsf{G}$ is simple, by \cite[Theorem~6.2]{Onishchik3}, the only possible extension is of type I, and thus the Lie algebra of the isometry group of an $\mathsf{SO}_{n+2}$--invariant metric on $M$ is isomorphic to either $\mathfrak{so}_{n+2}$ or $\mathfrak{so}_{n+2}\oplus\mathfrak{so}_2$.

		Let us assume that the Lie algebra of $\Isom(M,\langle\cdot,\cdot\rangle_{\lambda,\mu_1,\mu_2})$ is isomorphic to $\mathfrak{so}_{n+2}\oplus\mathfrak{a}$, where $\mathfrak{a}\cong\mathfrak{so}_2$ and let $V$ be the Killing vector field induced by $\mathfrak{a}$. Without loss of generality, we can assume that $V_o\neq 0$ by the effectiveness of the isometry group action. Moreover, notice that the isotropy representation of $M=\mathsf{SO}_{n+2}/\mathsf{SO}_{n}$  fixes the $1$-dimensional subspace $\mathfrak{m}_0$ of $\mathfrak{m}$ which is spanned by $v:=E_{1,2}$.  Since $[\mathfrak{so}_{n+2},\mathfrak{a}]=0$, we have $g_{*} V_o=V_{g\cdot o}$ for every $g\in\mathsf{SO}_{n+2}$, and thus, $V$ is $\mathsf{SO}_{n+2}$--invariant and also of constant length.
		Moreover, $V$ is also $\mathsf{SO}_{n}$--invariant. Because any $\mathsf{SO}_n$--invariant vector field must have a value at the origin corresponding to a vector in $\mathfrak{m}_0=\mathsf{Z}_{\mathfrak{m}}(\mathfrak{h})$, and $\mathfrak{m}_0$ is $1$-dimensional, $V_o$ must be proportional to $v^*_o$. Now, for each $X\in\mathfrak{m}$, we have
		$(\nabla_{X^{*}} V)_o=(\nabla_{X^*} - \nabla_{X^*}^c)_{o} V=D_{X} v$,
		where we have used  that every $\mathsf{SO}_{n+2}$--invariant tensor is parallel under the canonical connection $\nabla^c$; and that $D$ is a $(1,2)$-tensor. Now if we take $X=E_{13}+ E_{23}$, as $V$ is Killing we can compute
		\begin{equation}
			\label{eq:compkillingstiefel}
			0=\langle(\nabla_{X^*} V)_o,X^*\rangle_{\lambda,\mu_1,\mu_2}=\langle D_{X} v,X\rangle_{\lambda,\mu_1,\mu_2}=\mu_2-\mu_1. 
		\end{equation}
		Consequently, if the Lie algebra of $\Isom(M,\langle\cdot,\cdot\rangle_{\lambda,\mu_1,\mu_2})$ is isomorphic to $\mathfrak{so}_{n+2}\oplus\mathfrak{so}_2$, we have $\mu_1=\mu_2$. 
		
		Conversely, let us assume that $\mu_1=\mu_2$. Now we define the $\mathsf{G}$--invariant vector field $V$ such that $V_o=v^*_o$, where $v:=E_{1,2}\in\mathfrak{m}_0$. Now observe that for each $X\in\mathfrak{m}$, we have
		\begin{equation}
			\label{eq:compkillingstiefel2}
			\begin{aligned}
				\langle D_X V, X\rangle_{\lambda,\mu_1,\mu_2}&=\langle D_{X_{\mathfrak{m}_0}} V, X_{\mathfrak{m}_0}\rangle_{\lambda,\mu_1,\mu_2} + \langle D_{X_{\mathfrak{m}_0}} V, X_{\mathfrak{m}_1}\rangle_{\lambda,\mu_1,\mu_2} + \langle D_{X_{\mathfrak{m}_0}} V, X_{\mathfrak{m}_2}\rangle_{\lambda,\mu_1,\mu_2}\\
				&+\langle D_{X_{\mathfrak{m}_1}} V, X_{\mathfrak{m}_0}\rangle_{\lambda,\mu_1,\mu_2} + \langle D_{X_{\mathfrak{m}_1}} V, X_{\mathfrak{m}_1}\rangle_{\lambda,\mu_1,\mu_2} + \langle D_{X_{\mathfrak{m}_1}} V, X_{\mathfrak{m}_2}\rangle_{\lambda,\mu_1,\mu_2}\\
				&+\langle D_{X_{\mathfrak{m}_2}} V, X_{\mathfrak{m}_0}\rangle_{\lambda,\mu_1,\mu_2} + \langle D_{X_{\mathfrak{m}_2}} V, X_{\mathfrak{m}_1}\rangle_{\lambda,\mu_1,\mu_2} + \langle D_{X_{\mathfrak{m}_2}} V, X_{\mathfrak{m}_2}\rangle_{\lambda,\mu_1,\mu_2}\\
				&=(\mu_2-\mu_1)b(J X_{\mathfrak{m}_1}, X_{\mathfrak{m}_2}), 
			\end{aligned}
		\end{equation}
		where $X\in\mathfrak{m}$ and we have used Equations~\eqref{eq:dtensorstiefel0} and~\eqref{eq:dtensorstiefeln0}. 
		Since $V$ is $\mathsf{G}$--invariant, we have that $\langle \nabla_{X^*} V, X^*\rangle_{\lambda,\mu_1,\mu_2}=\langle D_{X} V, X\rangle_{\lambda,\mu_1,\mu_2}$, for each $X\in\mathfrak{m}$. Hence, Equation~\eqref{eq:compkillingstiefel2} implies that $V$ is a Killing vector field when $\mu_1=\mu_2$. Consequently, the Lie algebra of 
		$\Isom(M,\langle\cdot,\cdot\rangle_{\lambda,\mu_1,\mu_2})$ is isomorphic to $\mathfrak{so}_{n+2}\oplus\mathfrak{so}_2$, yielding the desired result.
	\end{proof}
	
	Now we are in a position to compute the index of symmetry of the unit tangent bundle of $\mathsf{S}^{n+1}$ with an $\mathsf{SO}_{n+2}$--invariant metric.
	\begin{theorem}
		\label{th:indstiefel}
		The index of symmetry of an $\mathsf{SO}_{n+2}$--invariant metric on the smooth manifold $M=T_1 \mathsf{S}^{n+1}=\mathsf{SO}_{n+2}/\mathsf{SO}_n$ with $n\ge3$ is given by
		\begin{equation}
			\begin{aligned}
				i_{\mathfrak{s}}(M,\langle\cdot,\cdot\rangle_{\lambda,\mu_1,\mu_2}) =  \left\{ \begin{array}{ll}
					1 \quad &\textrm{when $\mu_1=\mu_2$, or $\lambda=\mu_1+\mu_2$,} \\ \smallskip
					n \quad &\textrm{when $\lambda=\rvert \mu_1-\mu_2\rvert$,}  \\ \smallskip
					0 \quad &\textrm{otherwise}.
				\end{array} \right.
			\end{aligned}
		\end{equation}
		
	\end{theorem}
	\begin{proof}
		First of all, notice that the distribution of symmetry $\mathfrak{s}$ at the base point $o\in M$ is identified with an isotropy invariant subspace of $\mathfrak{m}$. By Lemma~\ref{lem:isomstiefel}, $\mathfrak{s}_o$ must be a proper subspace of $\mathfrak{m}$ as otherwise $M$ would be carrying a symmetric metric yielding a contradiction. Let us define \begin{equation}
			\label{eq:invG2sub}
			V_{\alpha,\beta}:=\{ \alpha v + \beta J v: v\in \mathfrak{m}_1\}, \quad \text{for some $\alpha,\beta\in\R$ satisfying $\alpha^2+\beta^2\neq 0$.} 
		\end{equation}	
		As $\mathfrak{m}_1$ and $\mathfrak{m}_2$ are equivalent irreducible $\mathsf{SO}_n$-modules of real type, we deduce that every proper non-trivial  $\mathsf{SO}_n$--invariant subspace of  $\mathfrak{m}$ is equal to
		\[\mathfrak{m}_0, \quad \mathfrak{m}_0\oplus V_{\alpha,\beta}, \quad V_{\alpha,\beta}, \quad \text{or} \quad  \mathfrak{m}_1\oplus\mathfrak{m}_2.\]
		Notice that by~\cite[Theorem~5.3]{BOR}, $\mathfrak{s}\neq \mathfrak{m}_1\oplus\mathfrak{m}_2$ as there are no homogeneous metrics with coindex equal to one on $M$. Now we prove that $V_{\alpha,\beta}\not\subset \mathfrak{s}$ when $\alpha\beta\neq0$. For the sake of contradiction, let us assume that $\alpha,\beta\neq0$ and that  $\xi\in V_{\alpha,\beta}\setminus\{0\}$ given by $\xi:=\alpha X_1 +\beta J X_1$, where $X_1\in\mathfrak{m}_1$ is a unit vector with respect to $b$. By Equations~\eqref{eq:nablastiefel1} and \eqref{eq:nablastiefel2}, we deduce
		\begin{equation}
			\label{eq:transvecstiefel1}
			\langle \nabla_{E^*_{1,2}} \xi^*,X^{*}_{1}\rangle_{\lambda,\mu_1,\mu_2}=\frac{\beta}{2}(\mu_2-\mu_1-\lambda), \quad \langle \nabla_{E^*_{1,2}} \xi^*,JX^{*}_{1}\rangle_{\lambda,\mu_1,\mu_2}=\frac{\alpha}{2}(\mu_2-\mu_1+\lambda).
		\end{equation}
		This proves that $V_{\alpha,\beta}\not\subset \mathfrak{s}$ when $\alpha\beta\neq0$, as otherwise we would have $\lambda=0$ which yields a contradiction. Therefore, the only possibilities for $\mathfrak{s}$ are
		\[\mathfrak{m}_0, \quad \mathfrak{m}_0\oplus\mathfrak{m}_i, \quad \mathfrak{m}_i, \quad \text{where $i\in\{1,2\}$}. \]
		
		On the one hand, if we assume that $\mu_1\neq\mu_2$, Equations~\eqref{eq:nablastiefel0}, \eqref{eq:nablastiefel1}, and~\eqref{eq:nablastiefel2}, imply that $\mathfrak{m}_1$ gives rise to a leaf of symmetry if and only if $\lambda=\mu_1-\mu_2$; and $\mathfrak{m}_2$ gives rise to a leaf of symmetry if and only if $\lambda=\mu_2-\mu_1$. Moreover, by  Equations~\eqref{eq:nablastiefel0}, \eqref{eq:nablastiefel1}, and~\eqref{eq:nablastiefel2}, $E_{1,2}^*$ is an infinitesimal transvection at $o$ if and only if $\lambda=\mu_1+\mu_2$. Consequently, the distribution of symmetry at $o$ is $\mathfrak{m}_i$ when $\lambda=\mu_i-\mu_{i+1}$, where $i\in\{1,2\}$ is taken modulo $2$; and $\mathfrak{m}_0$ when $\lambda=\mu_1+\mu_2$.
		
		On the other hand, let us assume that $\mu_1=\mu_2$. By Lemma~\ref{lem:isomstiefel}, the algebra of Killing vector fields of $(M,\langle\cdot,\cdot\rangle_{\lambda,\mu_1,\mu_2})$ is equal to $\mathfrak{so}_{n+2}\oplus\mathfrak{so}_2$. Notice that in this case,  $\mathfrak{m}_i\not\subset\mathfrak{s}$~by  Equations~\eqref{eq:nablastiefel0}, \eqref{eq:nablastiefel1}, and~\eqref{eq:nablastiefel2}, as it would imply that $\lambda=0$. Now consider the vector field $\eta:=c_1 E^*_{1,2} + c_2 V$, where $V$ is the Killing vector field induced by the $\mathfrak{so}_2$ ideal in the isometry group of  $(M,\langle\cdot,\cdot\rangle_{\lambda,\mu_1,\mu_2})$, that is, the $\mathsf{SO}_{n+2}$--invariant vector field such that $V_o=(E^*_{1,2})_o$. We now compute $c_1$ and $c_2$ so that $\eta$ is an infinitesimal transvection of $(M,\langle\cdot,\cdot\rangle_{\lambda,\mu_1,\mu_2})$. Notice that $\langle\nabla_{X^*} \eta, Y^*\rangle_{\lambda,\mu_1,\mu_2}=c_1\langle\nabla_{X^*} E^*_{1,2}, Y^*\rangle_{\lambda,\mu_1,\mu_2} + c_2 \langle D_{X^*} E_{1,2}, Y^*\rangle_{\lambda,\mu_1,\mu_2}$. By Equations~\eqref{eq:dtensorstiefel0}, \eqref{eq:dtensorstiefeln0}, \eqref{eq:nablastiefel1}, \eqref{eq:nablastiefel0}, \eqref{eq:nablastiefel1}, and~\eqref{eq:nablastiefel2}, this implies that $\eta$ is an infinitesimal transvection if and only if $c_1(2\mu_1-\lambda) -  c_2 \lambda=0$. Hence, if $\lambda=2\mu_1$, then we take $c_1=1$ and $c_2=0$, and $\eta$ is a transvection; and if $\lambda\neq2\mu_1$, we take $c_1=1$ and $c_2=\frac{2\mu_1-\lambda}{\lambda}$, so $\eta$ is also an infinitesimal transvection.\qedhere
	\end{proof}
	Now we consider the homogeneous fibrations induced by the triples \[ \textup{(1)} \enspace \mathsf{H}=\mathsf{SO}_n<\mathsf{K}=\mathsf{SO}_{n+1}<\mathsf{G}=\mathsf{SO}_{n+2}, \quad \textup{(2)} \enspace \mathsf{H}=\mathsf{SO}_n<\mathsf{K}=\mathsf{SO}_n\times\mathsf{SO}_2<\mathsf{G}=\mathsf{SO}_{n+2}.\]
	Observe that the first triple induces an $\mathsf{S}^n$--fibration, and the second an $\mathsf{S}^1$--fibration.
	We denote by $\langle\cdot,\cdot\rangle_{\lambda}$ and $\langle\langle\cdot,\cdot\rangle\rangle_{\lambda}$ the canonical variations with respect to the normal homogeneous metric of $T_1 \mathsf{S}^{n+1}=\mathsf{SO}_{n+2}/\mathsf{SO}_n$ associated with the homogeneous fibrations \textup{(1)} and \textup{(2)}, respectively. 
	
	In view of Theorem~\ref{th:indstiefel},  we see two very different behaviors. When we consider the $\mathsf{S}^1$--homogeneous fibration, all canonical variations have index of symmetry equal to $1$, while when we consider homogeneous fibrations with fiber $F$ of dimension bigger than one, we see that there is just one special canonical variation metric, namely the one with vertical deformation parameter $\lambda=2$, such that the leaf of symmetry coincides with $F$.
	This phenomenon holds more generally, see~Proposition~\ref{prop:s1bundlelow} and Proposition~\ref{prop:verticaltrans}.

	\subsection{The structure of the distribution of symmetry for $2$-isotropy irreducible homogeneous spaces}
	\label{subsec:structure2isotr}
	Throughout this subsection, we will focus on homogeneous spaces whose isotropy representation exhibits a particularly simple decomposition. Let $M=\mathsf{G}/\mathsf{H}$ be a Riemannian homogeneous space, and let $\mathfrak{g} = \mathfrak{h} \oplus \mathfrak{m}$ be a reductive decomposition. We say that $M$ is a \emph{$2$-isotropy irreducible homogeneous space} if the tangent space $\mathfrak{m} \cong T_o M$ decomposes into exactly two inequivalent irreducible submodules under the action of the connected component of the isotropy group $\mathsf{H}^{0}$ (or equivalently, under the adjoint action of the isotropy subalgebra $\mathfrak{h}$).
	
	We start this subsection by proving a lemma which gives a structural result for the Killing vector fields spanning the leaves of symmetry of an arbitrary Riemannian homogeneous space.
	\begin{lemma}
		\label{lemma:jacobisym}
		Let $M=\mathsf{G}/\mathsf{H}$ be a Riemannian homogeneous space with index of symmetry equal to $\ell$, and let $L(o)$ be its leaf of symmetry passing through the base point $o\in M$. Assume that $\gamma\colon \R\rightarrow L(o)$ is a geodesic of $L(o)$ such that $\gamma(0)=o$, $X\in\mathfrak{g}$ satisfies that $X^*_{\gamma(t)}\in T_{\gamma(t)} L(o)$, and $\{e_i\}^\ell_{i=1}$ is an orthonormal basis of $T_o L(o)$ diagonalizing the Jacobi operator $R_{\dot{\gamma}(0)}$ of $M$ restricted to $T_o L(o)$, chosen such that $e_1 = \dot{\gamma}(0)$, with corresponding eigenvalues $\{\alpha_i \}^\ell_{i=1}\subset\R$.
		
		Then, $X^*_{\gamma(t)}=\sum^\ell_{i=1} \lambda_i(t) e_i(t)$, where for each $i\in\{1,\ldots, \ell\}$, we have $e_i(t)=\mathcal{P}_{\gamma(t)} e_i$. Moreover, $\lambda_1(t) = \langle X^*_o, \dot{\gamma}(0)\rangle$ is constant, and for each $i\in\{2,\ldots, \ell\}$:
		\[\lambda_i(t)=\begin{cases} \frac{\langle D_{t_{\rvert t=0}}X^*, \, e_i\rangle}{\sqrt{\alpha_i}} \sin(\sqrt{\alpha_i}t) + \langle X^*_o,  e_i\rangle \cos(\sqrt{\alpha_i}t) & \text{when $\alpha_i>0$,}\\[2ex]
			\langle D_{t_{\rvert t=0}}X^*, e_i\rangle t +  \langle X^*_o, e_i\rangle & \text{when $\alpha_i=0$,}\\[2ex]
			\frac{\langle D_{t_{\rvert t=0}}X^*, \, e_i\rangle}{\sqrt{-\alpha_i}} \sinh(\sqrt{-\alpha_i}t) + \langle X^*_o, e_i\rangle \cosh(\sqrt{-\alpha_i}t) & \text{when $\alpha_i<0$.}
		\end{cases}  \]
	\end{lemma}
	\begin{proof}
		Let $X\in\mathfrak{g}$ and consider the induced Killing vector field $X^*$. As $X^*$ is a Killing vector field and $L(o)$ is a totally geodesic submanifold of $M$, we have $0=\langle \nabla_Y X^*, Y\rangle = \langle \nabla_Y (X^*)^{TL(o)}, Y\rangle$ for each $Y\in\mathfrak{X}(L(o))$, where $(X^*)^{TL(o)}$ is the tangential projection of the restriction of $X^*$ to $L(o)$. Thus, $(X^*)^{TL(o)}$ is a Killing vector field of $L(o)$. In particular, $(X^*)^{TL(o)}$ is a Jacobi vector field along the geodesic $\gamma$. Hence, as $X^*_{\gamma(t)}\in T_{\gamma(t)}L(o)$, we have  
		\begin{equation}
			\label{eq:jacobieq}
			D^2_t X^* + R_{\dot{\gamma}(t)} X^{*}=0.
		\end{equation}
		On the other hand, as $X^*$ is a Killing vector field and $\gamma$ is a geodesic, the inner product $\langle X^*_{\gamma(t)}, \dot{\gamma}(t) \rangle$ is constant. Thus, the component of $X^*$ along $e_1(t) = \dot{\gamma}(t)$ is simply $\lambda_1(t) = \langle X^*_o, \dot{\gamma}(0)\rangle$.
		
		Let $\{e_i(t)\}^\ell_{i=1}$ be a parallel frame along $\gamma$ such that $\{e_i(0)\}^\ell_{i=1}$ is a family of orthonormal vectors such that $R_{\dot{\gamma}(0)}e_i(0)=\alpha_i e_i(0)$ where $\alpha_i\in\R$, for each $i\in\{1,\ldots,\ell\}$. Then, as $L(o)$ is symmetric and totally geodesic, we have that $R_{\dot{\gamma}(t)}e_i(t)=\alpha_{i} e_i(t)$ for each $i\in\{1,\ldots,\ell\}$ and $t\in\R$. Hence, if we write $X^{*}_{\gamma(t)}=\sum_{i=1}^\ell \lambda_i(t) e_i(t)$, and we plug it in Equation~\eqref{eq:jacobieq}, we get $0=\lambda''_i(t) + \alpha_i \lambda_i(t)$ for each $i\in\{2,\ldots,\ell\}$. The solution to this system of linear ODEs, together with the constant component $\lambda_1(t)$, yields the desired result.
	\end{proof}
	Let us recall the concept of the \emph{nullity subspace} of a Riemannian manifold $ M $. This is defined as  
	\[
	\nu_{p} = \{v \in T_{p}M \colon R(x,y)v = 0\enspace \text{for every $x,y\in T_p M$}\}.
	\]  
	In the case of a Riemannian homogeneous manifold $ M = \mathsf{G}/\mathsf{H} $, since the curvature tensor $ R $ is $\s G $--invariant, the so-called \emph{nullity distribution} $ \nu = \bigcup_{p \in M} \nu_{p} $ is an autoparallel and flat distribution on $ M $. For more information on this distribution in homogeneous spaces, see~\cite{DOV}.
	
	Now we are in a position to prove the following structural result for the distribution of symmetry of a homogeneous space whose isotropy representation splits into the direct sum of two inequivalent irreducible submodules.
	\begin{theorem}\label{th:2-isotropy_fibration}
		Let $M = \mathsf{G}/\mathsf{H}$ be a compact 2-isotropy irreducible homogeneous space, equipped with an arbitrary $\mathsf{G}$--invariant non-symmetric metric. Moreover, let $\mathcal{D}_{1}$ and $\mathcal{D}_2$ be the $\mathsf{G}$--invariant distributions induced by the inequivalent $\mathsf{H}$-modules $\mathfrak{m}_1$ and $\mathfrak{m}_2$. Then, if $\mathfrak{s}$ is non-trivial we have:
		\begin{itemize}
			\item[\textrm{(i)}] $\mathfrak{s}=\mathcal{D}_{1}$ or $\mathfrak{s}=\mathcal{D}_{2}$. 
			\item[(ii)] The integral manifolds of $\mathfrak{s}$ are compact.
			\item[(iii)] Let $L(o)$ be the leaf of symmetry passing through $o$ and let
			\[
			\s K=\{g \in \mathsf{G}: g \cdot L(o) =L(o)\},
			\]
			that is, the stabilizer of $L(o)$ in $\mathsf{G}$. Then, $\mathsf{K}$ is a compact Lie subgroup of $\mathsf{G}$, and we have the homogeneous fibration
			\begin{equation*}\label{eq:HFibration}
				\mathsf{K}/\mathsf{H} \to M = \mathsf{G}/\mathsf{H} \overset{\pi}{\to} \mathsf{G}/\mathsf{K},
			\end{equation*}
			where $\mathsf{K}/\mathsf{H}$ is an isotropy irreducible symmetric space such that $T_{o}(\mathsf{K}/\mathsf{H})\cong\mathfrak{s}_{o}$, and the vertical distribution coincides with the distribution of symmetry. Moreover, for each $\mathsf{G}$--invariant metric on $M$, there is a unique $\mathsf{G}$--invariant metric on $\mathsf{G}/\mathsf{K}$ such that $\pi$ is a Riemannian submersion.
		\end{itemize}
	\end{theorem}
	\begin{proof}
		Let us start by proving \textrm{(i)}.  As we observed previously, $ \mathfrak{s} $ is invariant under the isotropy action, see~Equation~\eqref{eq:sym_isotropy}. Therefore, as $M$ is non-symmetric,  $\mathfrak{s}\neq0$, and  $\mathfrak{m}_1$ and $\mathfrak{m}_2$ are inequivalent irreducible $\mathsf{H}^{0}$-modules, we have $ \mathfrak{s} = \mathcal{D}_1 $ or $ \mathfrak{s} = \mathcal{D}_2 $. 
		
		Let us prove \textrm{(ii)}. We can assume without loss of generality, that $\mathfrak{s}=\mathcal{D}_{1}$. We will distinguish two cases depending on the rank  of the distribution $\mathcal{D}_{1}$.
		
		If $\rank(\mathcal{D}_{1})=1$, then $\mathsf{H}^0$ acts trivially on $\mathfrak{m}_{1}\cong(\mathcal{D}_{1})_{o}$. This implies that $L(o)$ corresponds to the connected component containing $o$ of the set of fixed points $M^{\mathsf{H}} = \{p \in M : h\cdot p = p \text{ for all } h \in \mathsf{H}\}$. This set is closed and, since it is contained in a compact space, it must also be compact.
		
		Now let us assume that $\dim(\mathcal{D}_{1})>1$. Let $L(o) = L_{0} \times \ldots \times L_{k}$ be the (local) de Rham decomposition of the leaf of symmetry passing through $o$, where $L_{0}$ is the flat factor of the leaf. As $L(o)$ is a totally geodesic submanifold of $M$, the nullity distribution $\nu$ of $L(o)$ at $o$ is equal to  $T_o L_0$. Moreover, for each $v\in T_o L_0$, using that the curvature tensor is $\mathsf{H}$--invariant, we have 
		\[R(x,y)hv=h R(h^{-1} x,h^{-1}y)v =0 \quad \text{for all $h\in \mathsf{H}$, and $x,y\in T_o L(o)$}.  \]
		Thus,  $T_{o}L_{0}$ is $\mathsf{H}^0$--invariant. Hence, by the irreducibility of $T_o L(o)\cong \mathfrak{m}_1$, we have that $T_{o}L_{0}$ is either trivial or coincides with $T_{o}L(o)$. Let us show that the latter cannot occur. If that were the case, we would have $T_{o}L_{0}=\mathfrak{s}_{o}$. Consider a non-zero $X \in \mathfrak{h}$ such that $\Exp(t X)\in \mathsf{H}^0$ does not act trivially on $T_oL_0$, since the $\mathsf{H}^0$--action on $T_{o} L_0$ does not have fixed vectors. Then, as $X^*_o=0$, by Equation~\eqref{eq:magicolmos}, we have
		\begin{equation}
			\label{eq:isotmagicolmos}
			e^{t(\nabla X^*)_{o}}v=d_o\phi_t(v)= d_o\Exp(t X) v. 
		\end{equation}
		Hence, we can find $v \in T_{o}L_{0}$ not fixed by $d_o\Exp(t X)$. Therefore, by Equation~\eqref{eq:isotmagicolmos}, we deduce that $(\nabla_{v} X^{*})_{o} \neq 0$. Now, if we consider the geodesic $\gamma_v(t)=\exp(tv)$, and by Lemma~\ref{lemma:jacobisym}, we have that 
		\[||X^*_{\gamma(t)}||^2=t^2 || D_{t_{\rvert t=0}}X^*_{\gamma(t)}||^2= t^2 || \nabla_{v} X^{*}_{o}||^2.\]
		However, this implies that the vector field $X^{*}$ along the geodesic $\gamma$ of $L_{0}$ has not bounded length contradicting the compactness of  $M$. Hence, the local de Rham decomposition of $L(o)$ cannot have local flat factors.
		
		Now, suppose that there exists some $ j $ such that $ L_j $ is of non-compact type.  Let $ v \in T_o L_j $, and let $ X \in \mathfrak{s} $ be a Killing field in $ L_j $ such that $ X_o = v $. We can choose $w$ such that the associated Jacobi operator has at least one strictly negative eigenvalue $\alpha_ {i} < 0$ corresponding to a direction $e_i$ where $\langle v, e_i \rangle \neq 0$. Since $ L_j $ is totally geodesic in $M$, the geodesic $ \gamma_w(t) = \exp(tw) $ remains in $ L_j $ for all $ t \in \R $. Now, since $X_{\gamma(t)} \in T_{\gamma(t)} L_j$, by Lemma~\ref{lemma:jacobisym}, we have that $X_{\gamma(t)} = \sum_{i=1}^{\ell_j} \lambda_i(t) e_i(t)$, where $e_i(t)$ is a parallel orthonormal frame along $\gamma_w$, and $\ell_j$ denotes the dimension of $L_j$. Since $L_j$ is a symmetric space of non-compact type, there is some $i\in\{1,\ldots, \ell_j\}$ and $\alpha_i<0$ such that $\lambda_i(t)$ satisfies 
		\begin{align*}
			\lambda_i(t)&=\frac{\langle D_{t_{\rvert t=0}}X^*, e_i\rangle}{\sqrt{-\alpha_i}} \sinh(\sqrt{-\alpha_i}t) + \langle X^*_o, e_i\rangle \cosh(\sqrt{-\alpha_i}t)=\langle X^*_o, e_i\rangle \cosh(\sqrt{-\alpha_i}t),
		\end{align*}
		where we have used the fact that $X\in\mathfrak{s}$ is an infinitesimal transvection  at $o\in M$. Now as the length of $X$ along $\gamma_w$ must be bounded we deduce that $\lambda_i(t)=\langle X^*_o, e_i\rangle=0$. However, this contradicts the fact that $L_j$ is contained in the leaf of symmetry, as there is no element in $\mathfrak{s}$ projecting non-trivially on $e_i\in T_o L_j$. Consequently, the de Rham decomposition of $L(o)$ can only have compact-type factors, and therefore it is compact, proving \textup{(ii)}. 
		
		Let us prove \textup{(iii)}. First, $\mathsf{K}$ is clearly an abstract subgroup of $\mathsf{G}$. Moreover, it is compact since the integral leaves of $\mathfrak{s}$ are compact by \textup{(ii)}, hence it is a closed Lie subgroup of $\mathsf{G}$. In addition, $\mathsf{K}$ acts transitively on $L(o)$, since $\mathsf{G}$ does and $\mathsf{K}$ is the maximal subgroup of $\mathsf{G}$ leaving $L(o)$ invariant. Therefore, $L(o)\cong \mathsf{K}/\mathsf{H}$. The map 
		\[
		\pi \colon \mathsf{G}/\mathsf{H} \longrightarrow \mathsf{G}/\mathsf{K}, \quad g\mathsf{H} \longmapsto g\mathsf{K},
		\]
		is clearly a smooth surjective submersion, whose fibers (diffeomorphic to $\mathsf{K}/\mathsf{H} \cong L(o)$) correspond to the integral leaves of the foliation defined by $\mathfrak{s}$.
		
		Let us now show that $\mathsf{K}/\mathsf{H}$ is an isotropy irreducible symmetric space. First, we have 
		\[
		T_{e\mathsf{H}}(\mathsf{K}/\mathsf{H}) \cong T_{o}\bigl(\pi^{-1}(e\mathsf{K})\bigr) \cong \mathfrak{s}_{o}.
		\]
		Since the $\mathsf{G}$--invariant metrics in this setting correspond to canonical variations of normal homogeneous metrics, it follows from Lemma~\ref{lemma:tghomfib} that the fibers are totally geodesic. Hence, the Levi-Civita connection of $M=\mathsf{G}/\mathsf{H}$ restricted to the fibers coincides with the Levi-Civita connection of the fibers themselves. Consequently, the vectors in $\mathfrak{s}_{0}$ act as transvections of $\pi^{-1}(e\mathsf{K}) \cong L(o)$ generating $T_{o}(\pi^{-1}(e\mathsf{K}))$, so $\mathsf{K}/\mathsf{H}$ is a symmetric space with respect to the restriction of the elements of $\mathsf{G}$ fixing $L(o)$, and hence with respect to their connected component $\mathsf{K}$. Furthermore, since $\mathfrak{s}_{o} \cong T_{e\mathsf{H}}(\mathsf{K}/\mathsf{H})$ is $\mathsf{H}^0$--irreducible, it follows that $\mathsf{K}/\mathsf{H}$ is an isotropy irreducible symmetric space.
		
		Finally, observe that since there are no equivalent isotropy submodules, the only $\mathsf{G}$--invariant metrics on $\mathsf{G}/\mathsf{H}$ are those arising from the canonical variation of a normal homogeneous metric on $\mathsf{G}/\mathsf{H}$, see Equation~\eqref{eq:canvar}. Thus, for each $\mathsf{G}$--invariant metric on $\mathsf{G}/\mathsf{H}$, there exists a unique $\mathsf{G}$--invariant metric on $\mathsf{G}/\mathsf{K}$ making $\pi$ into a Riemannian submersion. This completes the proof of \textup{(iii)}.\qedhere
	\end{proof}
	
	\subsection{The index of symmetry of $\mathsf{S}^1$--homogeneous fibrations}
	\label{subsec:s1fib}
	We start the section proving the following result which focuses on the case in which the fiber of a homogeneous fibration has dimension one.
	
	\begin{proposition}
		\label{prop:s1bundlelow}
		Let us assume that $M=\mathsf{G}/\mathsf{H}$, with $\mathsf{G}$ simple, is the total space of a homogeneous fibration induced by the triple of compact Lie groups $\mathsf{H}<\mathsf{K}<\mathsf{G}$ with $\dim\mathsf{K}/\mathsf{H}=1$. Consider the Riemannian homogeneous metric  $\langle\cdot,\cdot\rangle_{\lambda}$ on $M=\mathsf{G}/\mathsf{H}$  defined by Equation~\eqref{eq:lambdametric} and assume that it is not symmetric or isometric to a Riemannian product metric.	Then, if $\mathfrak{m}_2$ is an irreducible $\mathsf{H}$-module, we have
		\[i_{\mathfrak{s}}(\mathsf{G}/\mathsf{H},\langle\cdot,\cdot\rangle_{\lambda})= 1 \quad \text{for every $\lambda>0$}.\]  
	\end{proposition}
	\begin{proof}
		First of all, observe that $\mathfrak{m}_1$ is one-dimensional, and $\mathfrak{m}_2$ is irreducible. Thus, we have that $\mathcal{V}_o \cong \mathfrak{m}_1$ coincides with the set of fixed vectors of the action by $\mathsf{H}$. By Lemma~\ref{lemma:tghomfib}, the vertical distribution $\mathcal{V}$ is autoparallel, and hence integrable into a totally geodesic submanifold $L(o)$ of $M$ passing through $o \in M$. In fact, the integral manifold $L(o)$ of $\mathcal{V}$ coincides with the connected component through $o$ of the fixed point set in $M$ of the isotropy group $\mathsf{H}$. Note that $L(o) \cong \s{S}^1$, since it is an orbit of the compact Lie group $\mathsf{Z}_{\mathsf{G}}(\mathsf{H}):=\{g\in\mathsf{G}:gh=hg\}$, see~\cite[Lemma~A.1]{puttmann}.
		
		Now we will consider two vector fields whose value at $o$ spans the tangent space of the fiber $L(o)$. Firstly, let $V$ be the $\mathsf{G}$--invariant vector field on $M$ induced by a vector $v$ such that $\mathcal{V}_p = \mathbb{R}\, v$. In this way, the integral curves of $V$ parametrize the integral manifolds of $\mathcal{V}$. 
		Secondly, let $v^*$ be the Killing vector field induced by $v\in\mathfrak{m}_1$. We claim that $V$ and $v^*$ satisfy the following properties:
		\begin{enumerate}
			\item[\textup{(i)}] $V$ and $v^*$ are Killing vector fields with respect to every metric $\langle\cdot,\cdot\rangle_{\lambda}$.
			
			\item[\textup{(ii)}] $V_o$ and $v_o^*$ are $\mathsf{H}$--invariant vectors.
			
			\item[\textup{(iii)}] $(\nabla V)_o$ and $(\nabla v^*)_o$  are skew-symmetric $\mathsf{H}$--invariant linear endomorphisms of $\mathfrak{m}_2$.
			
			\item[\textup{(iv)}] $\nabla_v v^* = \nabla_v V = 0$.
			
			\item[\textup{(v)}] $[v^*, V] = 0$.
			
			\item[\textup{(vi)}] $[(\nabla V)_o, (\nabla v^*)_o]=0$.
			
			\item[\textup{(vii)}] $(\nabla V)_o \ne 0$.
		\end{enumerate}

		In order to prove $\textup{(i)}$, we just need to check that $V$ is a Killing vector field  with respect to every metric $\langle\cdot,\cdot\rangle_{\lambda}$. As a consequence of $\mathsf{G}$-invariance and Equation~\eqref{eq:LCconnectionkillingindarb}, we deduce that $V$ is a Killing field on $M$ with respect to the normal homogeneous metric, that is, with respect to $\langle\cdot,\cdot\rangle_1$, and thus the flow $\phi^V_t$ of $V$ is by isometries. Since the flow $\phi^V_t$ preserves the vertical distribution $\mathcal{V}$, it must also preserve the horizontal distribution $\mathcal{H}$. Furthermore, recall that the metrics $\langle\cdot,\cdot\rangle_{\lambda}$ are obtained by rescaling the normal homogeneous metric only in the direction of $\mathcal{V}$. Therefore, $\phi^V_t$ acts as an isometry for $\langle\cdot,\cdot\rangle_\lambda$, making $V$ a Killing field for every $\lambda > 0$. This proves~$\textup{(i)}$. 
		
		As $V$ is $\mathsf{G}$--invariant, it is also $\mathsf{H}$--invariant.  Since $v$ is a fixed vector of $\mathsf{H}$, then $[\mathfrak{h}, v] = 0$, and thus
		\[
		dh_o(v_o^*) = v_o^*, \quad \text{for all } h \in \mathsf{H},
		\]
		yielding~$\textup{(ii)}$. Now as a consequence of $\textup{(i)}$ and $\textup{(ii)}$, we have $\textup{(iii)}$. 	Let us consider the geodesic $L(o) \cong \s{S}^1$ of $(M,\langle\cdot,\cdot\rangle_{\lambda})$. Let $\phi^{v^*}_t$ be the flow of $v^*$. Since $v^*_o$ is tangent to $L(o)$ at $o$, and the integrable distribution $\mathcal{V}$ is $\mathsf{G}$--invariant, then $\phi^{v^*}_t(o)$ must always be tangent to $L(o)$. Then, since $\gamma(t) := \phi^{v^*}_{t}(o)$ has constant length, one has that $\gamma(t)$ is a geodesic. This implies that $\nabla_v v^* = 0$, and the same argument works if we replace $v^*$ by $V$, yielding $\textup{(iv)}$. Moreover, as $V$ is $\mathsf{G}$--invariant we have, $dg_p(V_p)=V_{gp}$ and thus $d\phi^{v^*}_t(V_p)=V_{\phi^{v^*}_t(p)}$.  Hence, $[v^*, V]_p=0$ for each $p\in M$ implying $\textup{(v)}$. Now by Equation~\eqref{eq:bracketkilling}, we also deduce $\textup{(vi)}$ as $[v^*,V]_o=0$ and $R_o(v^*, V)=0$. 
		Now, if $(\nabla V)_o=0$, by $\mathsf{G}$-invariance, $V$ would be a non-trivial parallel field and $M$ would split off a one-dimensional factor. This contradicts that $M$ is compact and simply connected. Hence, $\textup{(vii)}$ also holds. 
		
		By Schur's lemma, we have that the set of $\mathsf{H}$-intertwining maps $\mathrm{End}_{\mathsf{H}}(\mathfrak{m}_2)$ is isomorphic to an associative normed division algebra $\F$ equal to $\R$, $\C$ or $\H$. Moreover, the involution that maps $T\in\mathrm{End}_{\mathsf{H}}(\mathfrak{m}_2)$ to the adjoint map $T^*\in\mathrm{End}_{\mathsf{H}}(\mathfrak{m}_2)$ corresponds to the conjugation of the normed division algebra $\F$. Hence, skew-symmetric elements in  $\mathrm{End}_{\mathsf{H}}(\mathfrak{m}_2)$ correspond to purely imaginary elements in $\F$, implying by $(\textup{iii})$ and  $(\textup{vii})$, that  $\mathrm{End}_{\mathsf{H}}(\mathfrak{m}_2)$ is not isomorphic to~$\R$. Now two imaginary elements of $\H$ or $\C$ that commute must differ by a real multiple. Thus, by $(\textup{vi})$, we deduce that  $(\nabla v^*)_p = a(\nabla V)_o$ for some $a \in \mathbb{R}$. Then $\widetilde{V}:= v^* - aV$ is a non-trivial infinitesimal transvection of $(M,\langle\cdot,\cdot\rangle_{\lambda})$ at $o$. In fact, if $\widetilde{V}_o = 0$, since $(\nabla \widetilde{V})_o = 0$, then $\widetilde{V} = 0$. This would imply that $v^* = aV$, which is a contradiction with $\textup{(v)}$ since $\mathfrak{g}$ has a trivial center.  Finally, as $M$ is not locally isometric to a product or to a symmetric space, by Theorem~\ref{th:2-isotropy_fibration}, we deduce that 	$i_{\mathfrak{s}}(\mathsf{G}/\mathsf{H},\langle\cdot,\cdot\rangle_{\lambda})= 1$ for every $\lambda>0$ as we required.\qedhere
	\end{proof}
	
	Now we are in a position to prove Theorem~B.
	
	\begin{proof}[Proof of Theorem~B]
		Let $\mathsf{H}<\mathsf{K}<\mathsf{G}$ induce a homogeneous fibration with $B=\mathsf{G}/\mathsf{K}$ an irreducible symmetric space and $\mathsf{K}/\mathsf{H}=\mathsf{S}^1$. As $\mathfrak{h}$ is a codimension one subalgebra of the compact Lie algebra~$\mathfrak{k}$, it is an ideal. Then by the classification of irreducible symmetric spaces we have that $\dim(\s{Z}(\mathsf{K}))=1$; and  by \cite[Theorem~6.1]{Helgason}, we deduce that $B$ is an irreducible Hermitian symmetric space of compact type. Therefore $B$ is one of the spaces listed below:
		\[
		\mathsf{SU}_{p+q}/\mathsf{S}(\mathsf{U}_p\times \mathsf{U}_q),\ 
		\mathsf{SO}_{n+2}/(\mathsf{SO}_2\times \mathsf{SO}_n),\ 
		\mathsf{Sp}_n/\mathsf{U}_n,\ 
		\mathsf{SO}_{2n}/\mathsf{U}_n,\ 
		\mathsf{E}_6/(\mathsf{SO}_{10}\times \mathsf{U}_1),\ 
		\mathsf{E}_7/(\mathsf{E}_6\times \mathsf{U}_1)
		\]
		By inspection of the tables in~\cite{Dickinson-Kerr}, we deduce that all homogeneous $\mathsf{S}^1$--bundles over the aforementioned Hermitian symmetric spaces of compact type have two isotropy inequivalent modules except for $\mathsf{SO}_{n+2}/ \mathsf{SO}_n$, which is the $\mathsf{S}^1$--homogeneous bundle over the real Grassmannian of oriented $2$-planes.  As $B$ is an irreducible symmetric space which is not of group type, $\s{G}$ is simple. Thus, the result follows by Proposition~\ref{prop:s1bundlelow}.
	\end{proof}

	\subsection{The extension of vertical infinitesimal transvections }
	\label{subsec:transext}
	Now we focus on the problem of understanding when an infinitesimal transvection of the fiber $F$, with $\dim F>1$, is also a transvection in the total space of the homogeneous fibration.
	
	\begin{proposition}\label{prop:verticaltrans}
		Let us assume that $M=\mathsf{G}/\mathsf{H}$ is the total space of a homogeneous fibration induced by the triple of compact Lie groups $\mathsf{H}<\mathsf{K}<\mathsf{G}$. Consider the Riemannian homogeneous metric  $\langle\cdot,\cdot\rangle_{\lambda}$ on $M=\mathsf{G}/\mathsf{H}$  defined by the Equation~\eqref{eq:lambdametric} and let $\xi\in\mathfrak{m}_1$.  If $\xi^*_{\rvert\mathsf{K}/\mathsf{H}}$ is an infinitesimal transvection of the fiber $F\cong\mathsf{K}/\mathsf{H}$ at $o$, then $\xi^*$ is an infinitesimal transvection of $M=\mathsf{G}/\mathsf{H}$ at $o$ with respect to the metric $\langle\cdot,\cdot\rangle_{\lambda}$ for $\lambda=2$.
	\end{proposition}
	\begin{proof}
		Let us consider the metric $\langle \cdot,\cdot\rangle_{\lambda}$ on $M$ as in Equation~\eqref{eq:lambdametric}. Let $\xi\in\mathfrak{m}_1$ and assume that $\xi^*_{\rvert\mathsf{K}/\mathsf{H}}$ is an infinitesimal transvection of the fiber $F=\mathsf{K}/\mathsf{H}$ at $o$. By Lemma~\ref{lemma:tghomfib}, the vertical distribution $\mathcal{V}$ integrates to a totally geodesic submanifold of $M$ passing through $o$ with respect to $\langle\cdot,\cdot\rangle_{\lambda}$. Hence, as $\xi^*$ is Killing we have
		\[
		\langle ({\nabla}^{\lambda}_{v^*_{o}} \xi^*)_{o}, \eta^*_{o} \rangle_{\lambda} = - \langle ({\nabla}^{\lambda}_{\eta^*_{o}} \xi^*)_{o}, v^*_{o} \rangle_{\lambda} = 0,
		\]
		for all $\eta \in \mathfrak{m}_1$ and $v \in \mathfrak{m}_2$, where $\nabla^\lambda$ denotes the Levi-Civita connection associated with the metric $\langle\cdot,\cdot\rangle_{\lambda}$.

		Now, let $v,w \in \mathfrak{m}_2$. Using the general formula for the Levi-Civita connection for Killing fields (see ~Equation~\eqref{eq:L-C}), together with the definition of $\langle\cdot,\cdot\rangle_{\lambda}$, (see~Equation~\eqref{eq:lambdametric}), we obtain
		\begin{equation}
			\label{eq:LClambda}
			\begin{aligned}
				2\langle ({\nabla}^{\lambda}_{v^*_{o}}\xi^*)_{o}, w^*_{o}\rangle_\lambda 
				&= \langle [v^*, \xi^*]_{o}, w^*_{o}\rangle_\lambda + \langle [\xi^*, w^*]_{o}, v^*_{o}\rangle_\lambda + \langle [v^*, w^*]_{o}, \xi^*_{o}\rangle_\lambda \\
				&= -\langle [v, \xi], w\rangle_\lambda - \langle [\xi, w], v\rangle_\lambda - \langle [v, w], \xi \rangle_\lambda \\
				&= -b([v, \xi], w) - b([\xi, w], v) - \lambda b([v, w], \xi)= (2-\lambda)\,b([v, w], \xi),
			\end{aligned}
		\end{equation}
		where we have used that $\ad$ is skew-symmetric with respect to $b$, since $b$ is $\Ad(\mathsf{G})$--invariant. Therefore, by Equation~\eqref{eq:LClambda}, if $\lambda = 2$, the Killing vector field $\xi^*$ is an infinitesimal transvection of $M$ at $o$ with respect to the metric~$\langle\cdot,\cdot\rangle_2$, proving the desired result.\qedhere
	\end{proof}
	
	The following result gives a converse to the previous proposition in the case that $\dim F\ge2$.
	
	\begin{proposition}\label{prop:verticaltransconversedimge2}
		Let us assume that $M=\mathsf{G}/\mathsf{H}$ is the total space of a homogeneous fibration induced by the triple of compact Lie groups $\mathsf{H}<\mathsf{K}<\mathsf{G}$ with $\dim\mathsf{K}/\mathsf{H}\ge 2$. Consider the Riemannian homogeneous metric  $\langle\cdot,\cdot\rangle_{\lambda}$ on $M=\mathsf{G}/\mathsf{H}$  defined by the Equation~\eqref{eq:lambdametric} and suppose that it is not locally isometric to a Riemannian product metric.
		Furthermore, let us assume that the following conditions are satisfied:
		\begin{enumerate}
			\item[\textup{(i)}] the submodule $\mathfrak{m}_1$~is $\mathsf{H}^{0}$--irreducible,
			\item[\textup{(ii)}] the element $\xi\in\mathfrak{k}\setminus\{0\}$ induces an infinitesimal transvection on $\mathsf{K}/\mathsf{H}$, which is also   with respect to $\langle\cdot,\cdot\rangle_{\lambda}$.
		\end{enumerate}
		Then, if $\xi\in\mathfrak{m}_1$ or $(\mathsf{K},\mathsf{H})$ is a symmetric pair (not necessarily effective), we have that $\lambda=2$. 	
	\end{proposition}
	
	\begin{proof}
		Suppose that $\xi\in\mathfrak{k}\setminus\{0\}$ induces an infinitesimal transvection on the totally geodesic submanifold $\mathsf{K}/\mathsf{H}\subset\mathsf{G}/\mathsf{H}$ with respect to the induced holonomy irreducible metric $\langle\cdot,\cdot\rangle_{\lambda}$, and that  $\mathfrak{m}_1$ is $\mathsf{H}^{0}$--irreducible. Then, we can decompose $\xi\in\mathfrak{k}$ as $\xi=\xi_{\mathfrak{m}_1} + \xi_{\mathfrak{h}}$, where $\xi_{\mathfrak{m}_1}\in\mathfrak{m}_1$ and $\xi_{\mathfrak{h}}\in\mathfrak{h}$. We can define $\mathfrak{h}_0:=\mathsf{Z}_{\mathfrak{h}}(\mathfrak{m}_1)$. Notice that the connected subgroup $\mathsf{H}_0<\mathsf{H}$ with Lie algebra $\mathfrak{h}_0$ is a normal subgroup of $\mathsf{H}$ and induces Killing vector fields that vanish in fiber and that $\xi_{\mathfrak{h}}\in\mathfrak{h}_0$.
		
		We claim that if $\xi\in\mathfrak{m}_1$ or $(\mathsf{K},\mathsf{H})$ is a symmetric pair (not necessarily effective), we have $\xi_{\mathfrak{h}}\in\mathfrak{h}_0$. If $\xi\in\mathfrak{m}_1$, we have $\xi_{\mathfrak{h}}=0\in\mathfrak{h}_0$. Let us assume that $(\mathsf{K},\mathsf{H})$ is a symmetric pair (not necessarily effective). Observe that  $\xi^*_{\rvert\mathsf{K}/\mathsf{H}}$ and $\xi^*_{\mathfrak{m}_1\rvert\mathsf{K}/\mathsf{H}}$ are tangent to $\mathsf{K}/\mathsf{H}$, and they are Killing vector fields as they are the restriction of a Killing vector field to a totally geodesic submanifold. Furthermore, the vector field $\xi^*_{\mathfrak{m}_1\rvert\mathsf{K}/\mathsf{H}}$ is an infinitesimal transvection of $\mathsf{K}/\mathsf{H}$ at~$o$.  Thus,
		\[0= (\nabla \xi^*_{\rvert\mathsf{K}/\mathsf{H}})_o - (\nabla \xi^*_{\mathfrak{m}_1\rvert\mathsf{K}/\mathsf{H}})_o=(\nabla \xi^*_{\mathfrak{h}\rvert\mathsf{K}/\mathsf{H}})_o ,\]
		as the restriction of an infinitesimal transvection to a totally geodesic submanifold is an infinitesimal transvection of the totally geodesic submanifold.	
		Moreover, $\xi^*_{\mathfrak{h}\rvert\mathsf{K}/\mathsf{H}}$ is a Killing vector field of $\mathsf{K}/\mathsf{H}$ as it is the difference of two Killing vector fields on $\mathsf{K}/\mathsf{H}$. As $\xi_{\mathfrak{h}}\in\mathfrak{h}$,  we have $(\xi^*_{\mathfrak{h}\rvert\mathsf{K}/\mathsf{H}})_o=(\nabla \xi^*_{\mathfrak{h}\rvert\mathsf{K}/\mathsf{H}})_o=0$, and by uniqueness, we deduce that  $\xi^*_{\mathfrak{h}\rvert\mathsf{K}/\mathsf{H}}=0$ and therefore, $\xi_{\mathfrak{h}}\in\mathfrak{h}_0$ by Equation~\eqref{eq:L-C}. 
		
		Let us argue by contradiction. To this end, assume that $\lambda\neq 2$ and that $\xi^*$ is an infinitesimal transvection of $M$ at $o$ with respect to $\langle\cdot,\cdot\rangle_{\lambda}$. Take $v,w\in\mathfrak{m}_2$, then by Equation~\eqref{eq:L-C}, we have
		\begin{align*}
			2\langle (\nabla_{v^*_{o}}^\lambda \xi^*_{\mathfrak{h}})_{o}, w^*_{o}\rangle_{\lambda} &=\langle [v^*, \xi^*_{\mathfrak{h}}]_{o}, w^*_{o}\rangle_\lambda + \langle [\xi^*_{\mathfrak{h}}, w^*]_{o}, v^*_{o}\rangle_\lambda + \langle [v^*, w^*]_{o}, (\xi^*_{\mathfrak{h}})_{o}\rangle_\lambda\\
			&=	\langle [v^*, \xi^*_{\mathfrak{h}}]_{o}, w^*_{o}\rangle_\lambda +\langle [\xi^*_{\mathfrak{h}}, w^*]_{o}, v^*_{o}\rangle_\lambda=-b( [v, \xi_{\mathfrak{h}}], w) -b([\xi_{\mathfrak{h}}, w], v)\\
			&=2b( [v, w], \xi_{\mathfrak{h}}).
		\end{align*}
		Thus, by Equation~\eqref{eq:LClambda}, as $\xi^*$ is an infinitesimal transvection on $\mathsf{K}/\mathsf{H}$, we deduce the identity
		\begin{equation}
			\label{eq:conservedquantity}
			b([v, w], \xi_{\mathfrak{h}})=\frac{\lambda-2}{2} b([v, w], \xi_{\mathfrak{m}_1}) \quad \text{for all $v,w\in\mathfrak{m}_2$}.
		\end{equation}
		We claim that $[\mathfrak{m}_2,\mathfrak{m}_2]_{\mathfrak{m}_1}=\mathfrak{m}_1$. Notice that the subspace $[\mathfrak{m}_2,\mathfrak{m}_2]_{\mathfrak{m}_1}$ is an $\mathsf{H}^{0}$--invariant subspace of $\mathfrak{m}_1$ as the action of $\mathsf{H}^{0}$ on $\mathfrak{m}$ preserves the orthogonal splitting $\mathfrak{m}_1\oplus\mathfrak{m}_2$. Moreover, if $[\mathfrak{m}_2,\mathfrak{m}_2]_{\mathfrak{m}_1}$ is zero, then the horizontal distribution, which is $\mathsf{G}$--invariant and can be identified with $\mathfrak{m}_2$ at $o$, is integrable. This implies by~\cite{Di-Scala} that the metric $\langle\cdot,\cdot\rangle_{\lambda}$ is locally isometric to a Riemannian product metric contradicting our hypothesis. Thus, $[\mathfrak{m}_2,\mathfrak{m}_2]_{\mathfrak{m}_1}$ is a non-zero $\mathsf{H}^{0}$--invariant subspace of $\mathfrak{m}_1$, and by $\mathsf{H}^{0}$--irreducibility of $\mathfrak{m}_1$, we have $[\mathfrak{m}_2,\mathfrak{m}_2]_{\mathfrak{m}_1}=\mathfrak{m}_1$.
		
		Let us consider the linear function $\rho\colon\mathfrak{m}_1\rightarrow \R$ given by $\rho(\eta)=b(\eta,\xi_{\mathfrak{m}_1})$. We claim that $\rho$ is an $\mathsf{H}'$--invariant linear function, where $\mathsf{H}':=\mathsf{H}^{0}/\mathsf{H}_0$, that is, the effectivization of the $\mathsf{H}^{0}$ action on $\mathfrak{m}_1$.	As $[\mathfrak{m}_2,\mathfrak{m}_2]_{\mathfrak{m}_1}=\mathfrak{m}_1$, we can write every $\eta\in\mathfrak{m}_1$ as $\eta=[v,w]_{\mathfrak{m}_1}$, where $v,w\in\mathfrak{m}_2$. Then, for each $h\in \mathsf{H}'$, we have
		\begin{equation*}
			\begin{aligned}
				\rho(\Ad(h)\eta)&=	\rho(\Ad(h)[v,w]_{\mathfrak{m}_1})=\rho([\Ad(h)v,\Ad(h)w]_{\mathfrak{m}_1})=b([\Ad(h)v,\Ad(h)w],\xi_{\mathfrak{m}_1})\\
				&=\frac{2}{\lambda-2}b([\Ad(h)v,\Ad(h)w],\xi_{\mathfrak{h}})=\frac{2}{\lambda-2}b([v,w],\Ad(h^{-1})\xi_{\mathfrak{h}})=\frac{2}{\lambda-2}b([v,w],\xi_{\mathfrak{h}})\\
				&=b([v,w],\xi_{\mathfrak{m}_1})=b(\eta,\xi_{\mathfrak{m}_1})=\rho(\eta),
			\end{aligned}
		\end{equation*}
		where we have used Equation~\eqref{eq:conservedquantity}, the fact that $\xi_{\mathfrak{h}}\in\mathfrak{h}_0$ and that $b$
		is $\Ad(\mathsf{H}')$--invariant.
		
		Then, $\ker(\rho)=\mathfrak{m}_1\ominus\R\xi_{\mathfrak{m}_1}$, where $\ominus$ denotes the orthogonal complement with respect to $b$, is an invariant subspace of $\mathfrak{m}_1$ with codimension one. Consequently, this yields a contradiction with the fact that $\mathfrak{m}_1$ is an $\mathsf{H}^{0}$--irreducible module, and thus an $\mathsf{H}'$--irreducible module of dimension greater than one. \qedhere
	\end{proof} 
	Assume that $M=\mathsf{G}/\mathsf{H}$ is the total space of a homogeneous fibration with two inequivalent isotropy submodules.  The fiber $\mathsf{K}/\mathsf{H}$ is isotropy irreducible, and thus, by \cite[Corollary~1.3]{OR-crelle} the Lie algebra of $\mathsf{K}$ coincides with the Lie algebra of the isometry group of the fiber $F=\mathsf{K}/\mathsf{H}$ equipped with the metric induced by $\langle\cdot,\cdot\rangle$  up to its ineffective kernel. Then, as a consequence of Theorem~\ref{th:2-isotropy_fibration}, Proposition~\ref{prop:verticaltrans}, and Proposition~\ref{prop:verticaltransconversedimge2} we have the following result that allows us to compute the index of symmetry of many $\mathsf{G}$--invariant metrics on $2$-isotropy irreducible homogeneous spaces.
	
	\begin{corollary}
		\label{cor:indnotextension}
		Let $M=\mathsf{G}/\mathsf{H}$ be the total space of a homogeneous fibration equipped with a non-symmetric $\mathsf{G}$--invariant metric $\langle\cdot,\cdot\rangle$ and induced by the triple of compact Lie groups $\mathsf{H}<\mathsf{K}<\mathsf{G}$ where the fiber $F$ satisfies $\dim\mathsf{K}/\mathsf{H}\ge 2$. 
		Furthermore, let us assume that the following conditions are satisfied:
		\begin{enumerate}
			\item[\textup{(i)}] $(\mathsf{K},\mathsf{H})$ is a (not necessarily effective) symmetric pair or the fiber $F=\mathsf{K}/\mathsf{H}$ is neither a sphere nor a Lie group.
			\item[\textup{(ii)}] $\mathfrak{m}_1$ and $\mathfrak{m}_2$ are inequivalent $\mathsf{H}^{0}$--irreducible modules.
		\end{enumerate} 
		Then,  the leaf of symmetry of $(M,\langle\cdot,\cdot\rangle)$ coincides with $F$ if and only if $\langle\cdot,\cdot\rangle$ is homothetic to the metric $\langle\cdot,\cdot\rangle_{\lambda}$ with $\lambda=2$, defined in Equation~\eqref{eq:lambdametric}; and it is just a point in any other case.
	\end{corollary}
	We focus on those total spaces of the lists provided in~\cite{Dickinson-Kerr} that do not have spherical fibers.
	
	\begin{proof}[Proof of Theorem~C]
		The proof of the theorem is based on applying Corollary~\ref{cor:indnotextension} to the examples listed in the table. These are compact, simply connected homogeneous spaces $\mathsf{G}/\mathsf{H}$ with two inequivalent isotropy summands, where $\mathsf{G}$ is compact and simple, as classified by Dickinson and Kerr in~\cite{Dickinson-Kerr}. In all the examples considered below, the leaf of symmetry appears as the fiber of a homogeneous fibration associated with a triple of compact Lie groups $\mathsf{H} < \mathsf{K} < \mathsf{G}$, where the fiber $\mathsf{K}/\mathsf{H}$ is neither a sphere nor a Lie group, and is presented by a symmetric pair $(\mathsf{K}, \mathsf{H})$. Consequently, the infinitesimal transvections of $\s{K}/\s{H}$ at a point generate the tangent space to the fiber at that point. 
		
		It suffices to prove that, for $\lambda=2$, the metric cannot be symmetric on the total space. Thus, by Corollary~\ref{cor:indnotextension}, the leaf of symmetry is $F$, implying that the index of symmetry for that metric coincides with the dimension of the fiber $F$. To verify this, we first check whether the examples admit $\mathsf{G}$--invariant Einstein metrics. If they do not, then we claim that no $\mathsf{G}$--invariant metric can be symmetric. This is the case for the following examples
		\begin{align*}
			&\frac{\mathsf{E}_6}{\mathsf{Spin}_7  \mathsf{Spin}_3 \mathsf{SO}_2}, & &
			\frac{\mathsf{E}_6}{\mathsf{Spin}_5  \mathsf{Spin}_5  \mathsf{SO}_2}, & &
			\frac{\mathsf{E}_7}{\mathsf{Spin}_7 \mathsf{Spin}_5  \mathsf{Sp}_1},\\
			&\frac{\mathsf{E}_7}{\mathsf{Spin}_6  \mathsf{Spin}_6 \mathsf{Sp}_1}, & &
			\frac{\mathsf{E}_7}{\mathsf{Sp}_4 \times \mathsf{SO}_2}, & &
			\frac{\mathsf{E}_8}{\mathsf{SU}_8 \times \mathsf{Sp}_1},
		\end{align*}
		which correspond to examples IV.9, IV.11, IV.28, IV.35, IV.42, and IV.46 in the article~\cite{Dickinson-Kerr}. Indeed, suppose such a metric existed. If the symmetric metric were irreducible, it would necessarily be Einstein. If it were reducible, the simply connected compact homogeneous space $M$ would decompose as a Riemannian product $M=M_{1}\times \dots \times M_{k}$ of irreducible symmetric spaces of compact type. Since the isotropy representation decomposes into exactly two inequivalent irreducible summands ($\mathfrak{m}_1$ and $\mathfrak{m}_2$), the de Rham decomposition would be $M= M_1 \times M_2$, where the tangent spaces of the factors can be identified with $\mathfrak{m}_1$ and $\mathfrak{m}_2$. Irreducible symmetric spaces of compact type are Einstein with positive Einstein constants. Therefore, by an appropriate homothety on one of the factors (which corresponds to rescaling the metric along one of the irreducible modules), one can construct a product metric (and only one, up to homothety) which is a $\s{G}$--invariant Einstein metric.  However, as noted above, these specific examples do not admit any $\mathsf{G}$--invariant Einstein metrics. This contradiction implies that no $\mathsf{G}$--invariant metric can be symmetric.
		
		Now consider the families that do admit Einstein metrics, corresponding to examples I.18, II.6, and III.6 in~\cite{Dickinson-Kerr}:
		\begin{equation}\label{infinite_families}
			\mathsf{SO}_{m+2n}/(\mathsf{SO}_m\times \mathsf{U}_{n}), \enspace \mathsf{SU}_{n+m}/\mathsf{S}(\mathsf{SO}_{n}\times \mathsf{U}_{1}\times \mathsf{U}_{m}), \enspace \text{and}\enspace \mathsf{Sp}_{m+n}/(\mathsf{Sp}_{m}\times \mathsf{U}_{n}).
		\end{equation}
		First, we establish that if these spaces admitted a homogeneous symmetric metric, they would necessarily be irreducible symmetric spaces of compact type. Since $M$ is compact and simply connected, any symmetric metric must be of compact type. Furthermore, $M$ cannot be a product of symmetric spaces. 
		To see this, consider the reductive decomposition $\mathfrak{g}=\mathfrak{h}\oplus \mathfrak{m}_{1}\oplus \mathfrak{m}_{2}$, where $\mathfrak{m}_{1}$ and $\mathfrak{m}_{2}$ correspond to the vertical and horizontal subspaces, respectively. The intermediate subalgebra is $\mathfrak{k}=\mathfrak{h}\oplus \mathfrak{m}_{1}$. The irreducible module decomposition corresponds to the following homogeneous fibrations:
		\[
		\begin{aligned}
			& \mathsf{SO}_{2n}/\mathsf{U}_{n} \rightarrow \mathsf{SO}_{m+2n} / (\mathsf{SO}_{m}\times \mathsf{U}_{n}) \rightarrow \mathsf{SO}_{m+2n}/(\mathsf{SO}_{m}\times \mathsf{SO}_{2n}) \\
			& \mathsf{SU}_{n}/\mathsf{SO}_{n} \rightarrow \mathsf{SU}_{n+m}/\mathsf{S}(\mathsf{SO}_{n}\times \mathsf{U}_{1}\times \mathsf{U}_{m}) \rightarrow \mathsf{SU}_{n+m}/\mathsf{S}(\mathsf{U}_{n}\times \mathsf{U}_{m}) \\
			& \mathsf{Sp}_{n}/\mathsf{U}_{n} \rightarrow \mathsf{Sp}_{m+n}/(\mathsf{Sp}_{m}\times \mathsf{U}_{n}) \rightarrow \mathsf{Sp}_{m+n}/(\mathsf{Sp}_{m}\times \mathsf{Sp}_{n}).
		\end{aligned}
		\]
		In all cases, the base space is a symmetric space presented by a symmetric pair $(\mathsf{G}, \mathsf{K})$ of compact type. Since $\mathfrak{g}$ is simple, the decomposition $\mathfrak{g}=\mathfrak{k}\oplus \mathfrak{m}_{2}$ corresponds to a Cartan decomposition, implying $[\mathfrak{m}_{2},\mathfrak{m}_{2}] \subset \mathfrak{k}$. Specifically, because the base is an irreducible symmetric space of compact type, we have $[\mathfrak{m}_{2},\mathfrak{m}_{2}]_{\mathfrak{m}_{1}} \neq 0$ (in fact, it generates $\mathfrak{m}_1$).
		If the total space were a product of irreducible symmetric spaces, the tangent space would split into integrable distributions corresponding to the factors. Since the isotropy representation consists of inequivalent submodules $\mathfrak{m}_1$ and $\mathfrak{m}_2$, the splitting would have to align with these modules. However, the non-vanishing vertical component of the horizontal bracket shows that the distribution $\mathfrak{m}_2$ is not integrable. Thus, $M$ cannot be a Riemannian product.
		
		Consequently, if these spaces admitted a symmetric metric, they would be irreducible symmetric spaces of compact type. We can rule this out using classification results:
		\begin{itemize}
			\item Rank $=$ 1: None of the presentations in~\eqref{infinite_families} appear in the classification of transitive actions on CROSSes, see~\S~\ref{subsec:homCROSSes}, with the exception of the third family when $n=1$ and the space is diffeomorphic to a complex projective space, however, $n\ge3$ (see Table~\ref{table:theoremb}).
			\item Rank $>$ 1: According to \cite[Lemma 1.1]{kerr-einstein}, the only homogeneous presentation in these lists that admits an irreducible symmetric metric of rank strictly greater than 1 is the first family when $n=1$, but $n\ge2$ (see Table~\ref{table:theoremb}).
		\end{itemize}
		Therefore, by Corollary~\ref{cor:indnotextension}, the index of symmetry for the metric with $\lambda=2$ corresponds to the dimension of the fiber, and the leaf of symmetry is the fiber itself.
		\qedhere
	\end{proof}
	
	\section{The index of symmetry of homogeneous CROSSes}
	\label{sec:indexcross}
	In this section, we compute the index of symmetry of a homogeneous space diffeomorphic to a compact rank-one symmetric space. We start by considering the case of Hopf-Berger spheres.
	\begin{theorem}
		\label{th:hopfbergerindex}
		Let $M$ be isometric to a Hopf-Berger sphere $\mathsf{S}^{n}_{\F,\tau}$ with $\tau\neq1$. Then
		\[
		\begin{aligned}
			i_{\mathfrak{s}}(M) =  \left\{ \begin{array}{ll}
				1 \quad &\textrm{if $\mathbb{F}=\mathbb{C}$,} \\ \smallskip
				3 \quad &\textrm{if $\mathbb{F}=\mathbb{H}$ and $\tau = \frac{1}{2}$,}  \\ \smallskip
				7 \quad &\textrm{if $\mathbb{F}=\mathbb{O}$ and $\tau = \frac{1}{2}$,} \\ \smallskip
				0 \quad &\textrm{otherwise}.
			\end{array} \right.
		\end{aligned}
		\]
	\end{theorem}
	\begin{proof}
		
		First of all, let us consider the case of  $\mathsf{S}^{2n+1}_{\C,\tau}$. Notice that the isotropy representation splits into two irreducible modules, one of them of dimension one. Moreover, for every $\tau\neq 1$, $\mathsf{S}^{2n+1}_{\C,\tau}$ is not locally symmetric or locally isometric to a Riemannian product and $\mathrm{Isom}^0(\mathsf{S}^{2n+1}_{\C,\tau})=\mathsf{U}_{n+1}$, see Lemma~\ref{lem:Isom_Hopf_Berger}. Thus,  Proposition~\ref{prop:s1bundlelow}, implies that $i_{\mathfrak{s}}(\mathsf{S}^{2n+1}_{\C,\tau})=1$ for every $\tau\neq1$.

		Secondly, let us consider the case of $\mathsf{S}^{4n+3}_{\H,\tau}$. By~\cite[Theorem~3]{ziller-1977}, for each $\tau>0$, the Riemannian manifold $\mathsf{S}^{4n+3}_{\H,\tau}$ admits a naturally reductive decomposition $\mathfrak{g}=\mathfrak{h}\oplus\mathfrak{m}_\tau$ with respect to the presentation $\mathsf{G}/\mathsf{H}=\mathsf {Sp}_{n+1}\mathsf{Sp}_1/\mathsf{Sp}_{n}\mathsf{Sp}_1$. The Lie algebra of transvections with respect to $\nabla^c$ is given by $\tilde{\mathfrak{g}}=\mathfrak{m}_\tau\oplus[\mathfrak{m}_\tau,\mathfrak{m}_\tau]$, see~\cite[Theorem I.25]{kowalski}. Notice that this reductive decomposition is also naturally reductive for each $\tau>0$.
		The connected Lie group $\widetilde{\mathsf{G}}$ associated with the Lie algebra generated by the transvections with respect to $\nabla^c$ is a transitive normal connected subgroup  of   $\mathsf{Sp}_{n+1}\mathsf{Sp}_1$, see~\cite[Theorem I.25]{kowalski}. Then, it must be either   
		$\widetilde{\mathsf{G}}= \mathsf {Sp}_{n+1}\mathsf{Sp}_1$ or 
		$\widetilde{\mathsf{G}}=\mathsf{Sp}_{n+1}$. In the first case, the isotropy does not fix any non-zero vector, and thus $i_\mathfrak{s} (\mathsf{S}^{4n+3}_{\H,\tau})=0$. In the latter case, the  space is also naturally reductive with respect to the presentation $\mathsf{Sp}_{n+1}/ \mathsf{Sp}_{n}$. 
		However, we claim that the only naturally reductive metric is the normal homogeneous metric, up to scaling. Assume that there are two values $\tau_1>0$ and $\tau_2>0$, such that $\tilde{\mathfrak{g}}=\mathfrak{m}_{\tau_1}\oplus[\mathfrak{m}_{\tau_1},\mathfrak{m}_{\tau_1}]$ and $\tilde{\mathfrak{g}}=\mathfrak{m}_{\tau_2}\oplus[\mathfrak{m}_{\tau_2},\mathfrak{m}_{\tau_2}]$ are both naturally reductive decompositions.  Let us consider the Levi-Civita connection $\nabla^{\tau_i}$ of $\mathsf{S}^{4n+3}_{\H,\tau_i}$ for each $i\in\{1,2\}$. Observe that the tensor $\Theta$ given by $\Theta(X,Y)=\nabla_X^{\tau_1}Y-\nabla_X^{\tau_2} Y$ is symmetric as
		\begin{equation}
			\label{eq:Thetasym}
			\Theta(X,Y)=\nabla_X^{\tau_1}Y-\nabla_X^{\tau_2} Y = \nabla_Y^{\tau_1} X + [X,Y] - (\nabla_Y^{\tau_2} X + [X,Y]) = \nabla_Y^{\tau_1} X - \nabla_Y^{\tau_2} X = \Theta(Y,X), 
		\end{equation}
		since $\nabla^{\tau_i}$ is torsion-free for each $i\in\{1,2\}$. Moreover, given that $\widetilde{\mathsf{G}}= \mathsf {Sp}_{n+1}$ is simple, and $\widetilde{\mathfrak{g}}=\mathfrak{m}_1\oplus\mathfrak{m}_2\oplus\widetilde{\mathfrak{h}}$ as $\widetilde{\mathsf{H}}$-modules, we deduce that there exists a unique canonical connection. This implies that the geodesics of $\mathsf{S}^{4n+3}_{\H,\tau_1}$ and $\mathsf{S}^{4n+3}_{\H,\tau_2}$ are the same, and thus $\Theta$ is skew-symmetric. As $\Theta$ is symmetric by Equation~\eqref{eq:Thetasym}, we deduce that $\Theta=0$, and equivalently  $\nabla^{\tau_1}=\nabla^{\tau_2}$. Thus, the associated holonomy groups are equal, and since $\mathsf{S}^{4n+3}_{\H,\tau}$ has to be holonomy-irreducible, as it is simply connected and not diffeomorphic to a product, we deduce that the metrics are isometric, and therefore, $\tau_1=\tau_2$. As the normal homogeneous metric on $\mathsf{S}^{4n+3}_{\H,\tau}=\widetilde{\mathsf{G}}/\widetilde{\mathsf{H}}$ happens when $\tau=\tfrac{1}{2}$, see Remark~\ref{rem:normalhomtau}, we deduce that $i_{\mathfrak{s}}(\mathsf{S}^{4n+3}_{\H,\tau})=3$ if and only if $\tau=\tfrac{1}{2}$ and $0$  when $\tau\not\in\{1, \tfrac{1}{2}\}$.
		
		Thirdly, let us consider the case of $\mathsf{S}^{15}_{\mathbb{O},\tau}$. Notice that the isotropy representation splits into two irreducible inequivalent modules, and the vertical distribution has rank $7$. Moreover, for every $\tau\neq 1$, $\mathsf{S}^{15}_{\mathbb{O},\tau}$ is not locally symmetric or isometric to a Riemannian product and $\mathrm{Isom}^0(\mathsf{S}^{15}_{\mathbb{O},\tau})=\mathsf{Spin}_9$, see Lemma~\ref{lem:Isom_Hopf_Berger}. Thus,  Proposition~\ref{prop:verticaltransconversedimge2} implies that we can extend the vertical transvections to $\mathsf{S}^{15}_{\mathbb{O},\tau}$ if and only if $\tau=\tfrac{1}{2}$, which corresponds to the metric $\langle\cdot,\cdot\rangle$ with $\lambda=2$, see~Remark~\ref{rem:normalhomtau}. Consequently, $i_{\mathfrak{s}}(\mathsf{S}^{15}_{\mathbb{O},\tau})=7$ if $\tau=\tfrac{1}{2}$ and $0$ for every $\tau\not\in\{1, \tfrac{1}{2}\}$.
	\end{proof}
	
	Now we will study the index of symmetry of the Riemannian manifold $\mathsf{S}^{4n+3}_{\tau_1,\tau_2,\tau_3}$. Recall that this is the total space of the homogeneous fibration induced by $$\mathsf{H}=\mathsf{Sp}_n<\mathsf{K}=\mathsf{Sp}_n\mathsf{Sp}_1<\mathsf{G}=\mathsf{Sp}_{n+1}.$$
	Following the notation in \S~\ref{subsubsec:Spn-inv}, we consider the reductive decomposition $\mathfrak{g}=\mathfrak{h}\oplus\mathfrak{m}$, with $\mathfrak{m}:=\mathfrak{m}_0\oplus\mathfrak{m}_1$, where
	\begin{equation} \label{eq::red-desc-spn}
		\mathfrak{h}=\left(
		\begin{array}{c|c}
			Z & 0 \\
			\hline
			0 & 0
		\end{array}
		\right), \quad \mathfrak{m}_0=\left(
		\begin{array}{c|c}
			0 & 0 \\
			\hline
			0 & \mathrm{Im}(x)
		\end{array}
		\right), \quad \mathfrak{m}_1=\left(
		\begin{array}{c|c}
			0 & v \\
			\hline
			-v^* & 0
		\end{array}
		\right),
	\end{equation}
	where $x\in\H$, $v\in\H^n$ and $Z$ belongs to $\mathfrak{sp}_n$. Recall that we define $X_1$, $X_2$ and $X_3$ as the vertical Killing vector fields induced by vectors of $\mathfrak{m}_0$, with $x=i$, $x=j$ and $x=k$, respectively. Then, we express the metric of $\mathsf{S}^{4n+3}_{\tau_1,\tau_2,\tau_3}$ at the base point $o$ as
	\[
	\langle \cdot, \cdot \rangle_{\tau_1,\tau_2,\tau_3}= 2 \sum_{i=1}^{3}\tau_i b(\cdot,\cdot)_{\rvert \R X_i \times \R X_i}+ b(\cdot,\cdot)_{\rvert\mathfrak{m}_1\times\mathfrak{m}_1},
	\]
	where $b(X,Y)=-\tfrac{1}{2}\mathrm{Re} \tr_{\H}(XY)$, with $X,Y\in\mathfrak{g}$.
	
	The following lemma implies that the index of symmetry of $\mathsf{S}^{4n+3}_{\tau_1,\tau_2,\tau_3}$ is bounded by the index of symmetry of the fiber which is isometric to $\mathsf{S}^3_{\tau_1,\tau_2,\tau_3}$.
	
	\begin{lemma}
		\label{lem:transvS3red}
		Let $X$ be an infinitesimal transvection of $\mathsf{S}^{4n+3}_{\tau_1,\tau_2,\tau_3}$ at $o$. Then, the restriction of $X$ to the fiber $\mathsf{S}^3_{\tau_1,\tau_2,\tau_3}$ is an infinitesimal transvection of the fiber $\mathsf{S}^3_{\tau_1,\tau_2,\tau_3}$ at $o$. In particular,
		\begin{enumerate}
			\item[\textup{(i)}] $i_{\mathfrak{s}}(\mathsf{S}^{4n+3}_{\tau_1,\tau_2,\tau_3})=3$ if and only if $\tau_1=\tau_2=\tau_3\neq 1$;
			\item[\textup{(ii)}] $i_{\mathfrak{s}}(\mathsf{S}^{4n+3}_{\tau_1,\tau_2,\tau_3})\leq 1$ otherwise.
		\end{enumerate} 
	\end{lemma}
	\begin{proof}
		Let us denote the distribution of symmetry of $\mathsf{S}^{4n+3}_{\tau_1,\tau_2,\tau_3}$ by  $\mathfrak{s}$. Observe that $\Sp_n$ acts  trivially on $\mathfrak{m}_0$,  irreducibly on $\mathfrak{m}_1$ and the horizontal distribution   $\mathcal H$, which is identified with $\mathfrak{m}_1$ is completely non-integrable. Then $\mathfrak s\subset \mathcal V$. Equivalently, every infinitesimal transvection of $\mathsf{S}^{4n+3}_{\tau_1,\tau_2,\tau_3}$  at $o$ is tangent to the fiber $\mathsf{S}^3_{\tau_1,\tau_2,\tau_3}$.   By \cite[Lemma~A.1]{puttmann} note that the connected component passing through $o$ of the set of points fixed by $\mathsf{H}=\mathsf{Sp}_n$ is the orbit of the Lie group $\mathsf{N}_{\mathsf{Sp}_{n+1}}(\mathsf{Sp}_n)=\mathsf{Sp}_n\mathsf{Sp}_1=\mathsf{K}$, which is precisely the fiber $\mathsf{S}^{3}_{\tau_1,\tau_2,\tau_3}$. Thus, the fiber $\mathsf{S}^{3}_{\tau_1,\tau_2,\tau_3}$ is a totally geodesic submanifold of $\mathsf{S}^{4n+3}_{\tau_1,\tau_2,\tau_3}$, and Gauss formula implies that if $X\in\mathfrak{s}$, then $X$ is also an infinitesimal transvection of $\mathsf{S}^{3}_{\tau_1,\tau_2,\tau_3}$ at~$o$. 
		
		Finally, as the integral leaves of $\mathfrak{s}$ are totally geodesic and $\mathsf{S}^3_{\tau_1,\tau_2,\tau_3}$ admits totally geodesic surfaces if and only if it carries a round metric (see for instance \cite[Theorem~7.2]{tsukada}), we deduce that $i_{\mathfrak{s}}(\mathsf{S}^{4n+3}_{\tau_1,\tau_2,\tau_3})=3$ if and only if $\tau_1=\tau_2=\tau_3 \neq 1$, and $i_{\mathfrak{s}}(\mathsf{S}^{4n+3}_{\tau_1,\tau_2,\tau_3})\leq1$ in any other case.
	\end{proof}

	Now we will direct our attention to the case when $\tau_1>\tau_2>\tau_3>0$. 
	\begin{theorem}
		\label{th:sph3param}
		Let $M$ be isometric to $\mathsf{S}^{4n+3}_{\tau_1,\tau_2,\tau_3}$ with  $\tau_1>\tau_2>\tau_3>0$. Then, 
		\[\begin{aligned}
			i_{\mathfrak{s}}(M) =  \left\{ \begin{array}{ll}
				1 \quad &\textrm{if $\tau_{1}=1$ and $\tau_{2}=1-\tau_{3}$,} \\ \smallskip
				0 \quad &\textrm{otherwise.} \\
			\end{array} \right.
		\end{aligned}  \] 
	\end{theorem}
	\begin{proof}
		First of all, the infinitesimal transvections  of $\mathsf{S}^3_{\tau_1,\tau_2,\tau_3}$ are trivial when $\tau_1\neq\tau_2+\tau_3$, and are spanned by $X_1^*$ if and only if $\tau_1=\tau_2+\tau_3$, see~\cite[p.~35]{ORT}.
		Thus, by Lemma~\ref{lem:transvS3red},  	$i_{\mathfrak{s}}(\mathsf{S}^{4n+3}_{\tau_1,\tau_2,\tau_3})\leq1$ if $\tau_1=\tau_2+\tau_3$, and  $i_{\mathfrak{s}}(\mathsf{S}^{4n+3}_{\tau_1,\tau_2,\tau_3})=0$, otherwise. Thus, we will assume that $\tau_1=\tau_2+\tau_3$.
		
		Let $X\in\g{g}$ with $X^{*}_o=(X^*_1)_o$ and such that $(\nabla X^*)_o$ and $(\nabla X_1^*)_o$ are linear null endomorphisms of $T_o \mathsf{S}^{3}_{\tau_1,\tau_2,\tau_3}$.  The difference between $X^{*}$ and $X^*_1$ is a Killing vector field whose value at $o\in\mathsf{S}^{4n+3}$ is zero. Then,  there is $Z\in\mathfrak{h}=\mathfrak{sp}_n$ such that 
		\begin{equation}
			\label{eq:Zdef}
			X^*=X^*_1 + Z^*, \qquad (\nabla Z^*)_o=-(\nabla X_1^*)_o. 
		\end{equation}
		Let $v,w\in\mathfrak{m}_1$, then using Equation~\eqref{eq:L-C}, we have 
		\begin{equation}
			\label{eq:Xtrans}
			\begin{aligned}
				2\langle (\nabla_{v^*_{o}} X^*)_o, w^*_{o}\rangle&=\langle [v^{*},X^{*}]_{o},w^{*}_{o}\rangle + \langle [X^{*},w^{*}]_{o}, v^{*}_{o}\rangle + \langle [v^{*},w^{*}]_{o}, X^{*}_{o}\rangle\\
				&=-b([v,X],w) - b([X,w],v)-2\tau_1b([v,w],X_{1})\\
				&=-2(\tau_1-1) b([v,w],X_1) - 2 b([Z,w],v),
			\end{aligned}
		\end{equation}
		where we have used that $X=X_1+Z$ and that $b$ is $\Ad(\mathsf{G})$--invariant. By Equation~\eqref{eq:Xtrans}, $X^{*}$ is an infinitesimal transvection of $M$ at $o$ if and only if 
		\begin{equation}
			\label{eq:conservedqsp}
			b([Z,w],v)=(1-\tau_1)b([v,w],X_1) \qquad \text{for all $v,w\in\mathfrak{m}_1$}.
		\end{equation}
		Let $g\in\mathsf{Z}_{\mathsf{Sp}_{n+1}}(\mathsf{H})\cong\mathsf{Sp}_1$.Then, by Equation~\eqref{eq:conservedqsp}, we have
		\begin{equation}
			\label{eq:conservedqspg}
			\begin{aligned}
				b([Z,w],v)&=b(\Ad(g)[Z,w],\Ad(g)v)=b([Z,\Ad(g)w],\Ad(g)v)\\
				&=(1-\tau_1)b([\Ad(g)v,\Ad(g)w],X_1)=(1-\tau_1)b(\Ad(g)[v,w],X_1)\\
				&=(1-\tau_1)b([v,w],\Ad(g^{-1})X_1),
			\end{aligned}
		\end{equation} 
		where we have used that $Z\in\mathfrak{h}<\mathfrak{k}$, and $g\in\mathsf{Z}_{\mathsf{Sp}_{n+1}}(\mathsf{H})$.
		
		Now let us assume that $\tau_1\neq 1$. If there is a non-trivial infinitesimal transvection at $o$, by Lemma~\ref{lem:transvS3red}, it must be induced by a Killing vector field $X\in\g{g}$ satisfying the conditions in Equation~\eqref{eq:Zdef} and~\eqref{eq:conservedqspg}. On the one hand, we can always choose $g\in\mathsf{Sp}_1$ such that $b([v,w], \Ad(g^{-1}) X_1)=0$ as $\mathsf{Sp}_1$ acts transitively on the unit sphere of $\mathfrak{m}_0$. On the other hand, we can choose $w\in\mathfrak{m}_1$ such that $[Z,w]\neq 0$ since $\mathsf{H}$ acts faithfully on $\mathfrak{m}_1$. Then, by taking $v=[Z,w]$, the Equation~\eqref{eq:conservedqspg} implies
		\[0<b([Z,w],[Z,w])=b([v,w],\Ad(g^{-1}) X_1)=0\]
		yielding a contradiction with the fact that $\tau_1\neq1$.
		
		Now we assume that $\tau_1=1$. Then Equation~\eqref{eq:conservedqspg} implies that $b([Z,w],v)=0$ for every $v,w\in\mathfrak{m}_1$. However, as $\mathsf{H}$ acts faithfully on $\mathfrak{m}_1$, we deduce that $Z=0$, and by Equation~\eqref{eq:Zdef}, we prove that $X^*_1$ is an infinitesimal transvection of $\mathsf{S}^{4n+3}_{\tau_1,\tau_2,\tau_3}$.\qedhere
	\end{proof}

	\begin{remark}
		The metrics on the spheres $\mathsf{S}^{4n+3}_{\tau_{1},\tau_{2},\tau_{3}}$ with $\tau_1>\tau_2>\tau_3$ as described in the hypotheses of Theorem~\ref{th:sph3param} are not naturally reductive. Indeed, they are not even geodesic orbit spaces, i.e.\ their geodesics are not necessarily orbits. This can be seen as follows. If the ambient space were a geodesic orbit space, its totally geodesic fibers (which are isometric to $\mathsf{Sp}_{1}$) would necessarily inherit the geodesic orbit property. This holds because the subgroup of the full isometry group $\mathrm{Isom}^{0}(\mathsf{S}^{4n+3}_{\tau_{1},\tau_{2},\tau_{3}}) \cong \mathsf{Sp}_{n+1}$ that acts transitively on the fibers is precisely $\mathsf{Sp}_1$. 	Since $\mathsf{Sp}_{1}$ is a homogeneous space of dimension $3$, a result by Kowalski and Vanhecke~\cite{KoWa} implies that the geodesic orbit property is equivalent to natural reductivity. However, since $\mathsf{Sp}_{1}$ is a simple Lie group, its unique naturally reductive metric (up to scaling) with an isometry group locally isomorphic to $\mathsf{Sp}_{1}$ is the normal homogeneous metric induced by the Killing form. This corresponds exactly to the round metric on $\mathsf{S}^3$, which explicitly contradicts the fact that the induced metric on our fibers has three distinct parameters $\tau_1>\tau_2>\tau_3$.
	\end{remark}
	
	\begin{theorem}
		\label{th:sph2param}
		Let $M$ be isometric to $\mathsf{S}^{4n+3}_{\tau_1,\tau_2,\tau_3}$ with either $\tau_1>\tau=\tau_2=\tau_3\neq 1$ or $1\neq \tau_{1}=\tau_{2}=\tau>\tau_{3}$. Then,
		\[
		i_{\mathfrak{s}}(M) =  \left\{ \begin{array}{ll}
			1 \quad &\text{if } \tau=\frac{1}{2}, \\ \smallskip
			0 \quad &\text{otherwise.} \\
		\end{array} \right.
		\]
	\end{theorem}
	
	\begin{proof}
		Following the notation introduced in Section~\ref{subsubsec:Spn-inv}, we  study simultaneously the index of symmetry of the metrics satisfying $\tau_1>\tau=\tau_2=\tau_3\neq1$ and of those satisfying $1\neq\tau=\tau_1=\tau_2>\tau_3$, that is, Case (A) and Case (B), respectively.
		
		Let $X^*$ be an infinitesimal transvection at $o\in \mathsf{G}(q)/\mathsf{H}(q)$ with $X \in \mathfrak{g}(q)$. The restriction of $X^*$ to the fiber $\mathsf{S}^3_{\tau_1,\tau_{2},\tau_{3}}$ through $o$ must also be an infinitesimal transvection of $\mathsf{S}^3_{\tau_1,\tau_{2},\tau_{3}}$ by Gauss formula since the fiber is totally geodesic. The non-trivial transvection of the fiber $\mathsf{S}^3_{\tau_1,\tau_{2},\tau_{3}}$  is generated (up to scalar multiplication) by a linear combination of the vertical generator of $\Sp_{n+1}$ restricted to the fiber, $W(q)$ and $V(q)$, see~\cite[p.~35]{ORT}.   Specifically, this vector is given by $\xi(q) = W(q)+\alpha V(q)$, with $\alpha=1-\frac{2\tau}{\tau(q)}$, where $\tau(q)=\tau_{1}$ in Case (A) and $\tau(q)=\tau_{3}$ in Case (B).
		
		As the fiber is totally geodesic, a similar argument as in the second paragraph of the proof of Theorem~\ref{th:sph3param} implies that, up to scaling, $X=\xi(q)+Z$ for some $Z\in\g{h}(q)$ satisfying that $(\nabla Z^*)_o$ vanish identically when restricted to the tangent space of the fiber at $o$. By Equation~\eqref{eq:LCconnectionkillingind}, it follows that $[Z,Y]=0$ for every $Y\in\g{m}_0(q)\oplus\g{m}_1(q)$, which implies that $Z\in\g{h}_1\cong\g{sp}_n$. Also, we claim that
		\begin{equation}
			\label{eq:conservedqsp2}
			b([Z,w],v)=(1-2\tau)b([v,w],W(q)) \qquad \text{for all $v,w\in\mathfrak{m}_2$}.
		\end{equation}
		Let $v,w\in\mathfrak{m}_2$, then using Equation~\eqref{eq:LCconnectionkillingindarb},	we have				\begin{equation*}
			\label{eq:Xtrans2}
			\begin{aligned}
				0 &=2\langle (\nabla_{v^*_{o}} X^*)_{o}, w^*_{o}\rangle \\
				&= - b([v,X],w) - b([X,w],v)-4\tau(q)b([v,w],\xi(q)_{\mathfrak{m}(q)})\\
				&=-2 b([X,w],v) - 4\tau b([v,w], W(q))\\
				&=-2  b([Z,w],v)+ 2(1-2\tau) b([v,w],W(q))				\end{aligned}
		\end{equation*}
		where we have used that $[v,w]$ is orthogonal to $V(q)$ and that the orthogonal projection of $\xi(q)$ to $\g{m}(q)=\g{m}_0(q)\oplus\g{m}_1(q)\oplus\g{m}_2$ is given by	\[\xi(q)_{\mathfrak{m}(q)}= \xi(q)_{\mathfrak{m}_0(q)} = \frac{b(\xi(q), V(q)-W(q))}{b(V(q)-W(q), V(q)-W(q))} (V(q)-W(q)) = \frac{\tau}{\tau(q)}(W(q)-V(q)).\]
		Let $\s{H}_1\cong\s{Sp}_n$ be the connected subgroup of $\s{G}(q)$ with Lie algebra $\g{h}_1$. It turns out that $\s{H}_1$ acts transitively on the unit sphere of $\g{m}_2$, and thus an analogous argument applied to Equation~\eqref{eq:conservedqsp2} as in the last two paragraphs of the proof of Theorem~\ref{th:sph3param} allows us to conclude that $X^*$ is a non-zero infinitesimal transvection at $o$  if and only if $\tau=\tfrac{1}{2}$. By Lemma~\ref{lem:transvS3red}, there is at most one non-trivial infinitesimal transvection up to scalar multiplication,  yielding the desired result.
	\end{proof}
	
	\begin{remark}
		It can be checked that the metrics with non-trivial index of symmetry appearing in Theorem~\ref{th:sph2param} correspond precisely to the naturally reductive metrics on this family of spheres.
	\end{remark}
	
	\begin{theorem}
		\label{th:cpindex}
		Let $M$ be isometric to $\C\mathsf{P}^{2n+1}_{\tau}$ with $\tau\neq1$. Then 	$i_{\mathfrak{s}}(M) =0$.$ $
	\end{theorem}
	\begin{proof}
		Notice that the isotropy representation splits into two irreducible inequivalent modules, and the vertical distribution has rank $2$. Moreover, for every $\tau\neq 1$,  $\C\mathsf{P}^{2n+1}_{\tau}$ is not locally symmetric or isometric to a Riemannian product and $\mathrm{Isom}^0(\C\mathsf{P}^{2n+1}_{\tau})=\mathsf{Sp}_{n+1}/\mathsf{\Delta\Z_2}$, see Lemma~\ref{lem:Isometry_Group_Projective}. Thus,  Proposition~\ref{prop:verticaltransconversedimge2} implies that we can extend the vertical transvections to $\C\mathsf{P}^{2n+1}_{\tau}$  if and only if $\tau=1$, which is the metric corresponding to the metric $\langle\cdot,\cdot\rangle$ with $\lambda=2$, see~Equation~\eqref{eq:metricCPn} yielding  that $i_{\mathfrak{s}}(M) =0$ unless $\tau=1$.
	\end{proof}
	\begin{proof}[Proof of Theorem~A]
		Consequently, the proof of Theorem~A follows by combining Theorem~\ref{th:hopfbergerindex},  Theorem~\ref{th:sph3param}, Theorem~\ref{th:sph2param} and Theorem~\ref{th:cpindex}.
	\end{proof}

	\enlargethispage{2\baselineskip}
\end{document}